\documentclass{amsart}
\usepackage{amssymb,latexsym}
\theoremstyle{plain}
\newtheorem{theorem}{Theorem}

\newtheorem{proposition}{Proposition}
\newtheorem{lemma}{Lemma}
\theoremstyle{definition}
\newtheorem{definition}{Definition}
\newtheorem{example}{Example}
\newtheorem{notation}{Notation}
\newtheorem{remark}{Remark}

\DeclareMathOperator{\Res}{Res}
\DeclareMathOperator{\red}{red}
\DeclareMathOperator{\Sing}{Sing}

\newcommand{\enm}[1]{\ensuremath{#1}}          %
\newcommand{\cal}[1]{\mathcal{#1}}

\newcommand{\NN}{\enm{\mathbb{N}}}

\newcommand{\ZZ}{\enm{\mathbb{Z}}}

\newcommand{\PP}{\enm{\mathbb{P}}}

\newcommand{\Aa}{\enm{\cal{A}}}
\newcommand{\Bb}{\enm{\cal{B}}}

\newcommand{\Ff}{\enm{\cal{F}}}

\newcommand{\Hh}{\enm{\cal{H}}}
\newcommand{\Ii}{\enm{\cal{I}}}

\newcommand{\Oo}{\enm{\cal{O}}}

\renewcommand{\phi}{\varphi}
\renewcommand{\theta}{\vartheta}
\renewcommand{\epsilon}{\varepsilon}

\renewcommand{\to}[1][]{\xrightarrow{\ #1\ }}

\newcommand{\bb}{\mathfrak{b}}

\newcommand{\old}[1]{}

\date{}

\begin{document}

\title[curves]
{On the Hilbert function of general unions of curves in projective spaces}
\author{Edoardo Ballico}
\address{Dept. of Mathematics\\
 University of Trento\\
38123 Povo (TN), Italy}
\email{ballico@science.unitn.it}
\thanks{The author was partially supported by MIUR and GNSAGA of INdAM (Italy).}

\begin{abstract}
Let $X=X_1\cup \cdots \cup X_s\subset \PP^n$, $n\ge 4$, be a general union of smooth non-special curves with $X_i$ of degree
$d_i$ and genus $g_i$ and $d_i\ge \max \{2g_i-1,g_i+n\}$ if $g_i>0$. We prove that $X$ has maximal rank, i.e. for any $t\in
\NN$ either
$h^0(\Ii _X(t))=0$ or $h^1(\Ii_X(t))=0$.
\end{abstract}

\maketitle
\section{Introduction}
Let $X\subset \PP^n$ be a closed subscheme. $X$ is said to have \emph{maximal rank} if for every $t\in \NN$, the restriction map $H^0(\Oo _{\PP^n}(t))\to H^0(\Oo _X(t))$ has maximal rank, i.e. it is injective or surjective. Note that $X$ has maximal rank if and only if for each $t\in \NN$ either $h^0(\Ii _X(t)) =0$ or  $h^1(\Ii
_X(t))=0$. Since \cite{hh0,hi} there was a long quest to prove that `` general '' curves with fixed degree and genus in
$\PP^r$ have maximal rank (\cite{be01,be00,be2,be4,be7,be5}). For curves with general moduli the Maximal Rank Conjecture was
proved by E. Larson (\cite{l8,l9}). At least $4$ papers were devoted to general unions of simpler curves: general unions of
lines (\cite{hh0}), general unions of a rational curve of a prescribed degree and a prescribed number of lines in $\PP^3$
(\cite{hh3}), general unions of smooth rational curves in $\PP^n$, $n\ge 4$ (\cite{be6}) and general unions of smooth rational
curves in $\PP^3$ (\cite{b}).

In this paper we prove the following result.

\begin{theorem}\label{i1}
Fix integers $n\ge 4$, $s>0$, $g_i\ge 0$. If $g_i=0$ assume $d_i>0$. If $g_i\ge 1$ assume $d_i\ge \max \{2g_i-1,g_i+n\}$ . Let
$X = X_1\cup
\cdots
\cup
\subset
\PP^n$ be a general union of $s$ smooth curves with $\deg (X_i)=d_i$ and $p_a(X_i) =g_i$ for all $i$. Then $X$ has maximal rank.\end{theorem}

In Theorem \ref{i1} `` $X$ general '' means of course that the $s$-ple $(X_1,\dots ,X_s)$ is general in the product of the
$s$ irreducible components of the Hilbert scheme of $\PP^n$ containing $X_1,\dots ,X_s$. When
$s=1$ Theorem
\ref{i1} is true by a particular case of the Maximal Rank Conjecture for non-special curves (\cite{be4,be2,be5,l8}). We may
also quote
\cite{be7}, because the only case with $s=1$ and not covered by
\cite{be7} is the obvious case of rational normal curves of a proper linear subspace of $\PP^n$. For arbitrary $s$ the case
$g_1=\cdots =g_s=0$ of Theorem \ref{i1} is
\cite{be6}. As it is stated Theorem \ref{i1} would fail for $n=3$ and we give a list of the exceptional cases known
to us  (Remark \ref{a3} and Lemmas \ref{3ma1} and \ref{3ma2}), but we have not checked that the list contains all exceptional
cases. We believe that there are only finitely many exceptional cases
$(s;d_1,g_1;\dots ;d_s,g_s)$, even in the Brill-Noether range. All the exceptional cases we found are related to reducible
surfaces. In most exceptional cases the curves are only contained in reducible surfaces, but in one case, Example
\ref{3ma2.1}, the general quartic surface containing the general curve $X$ is irreducible. Even in this case the obvious
reducible quartics containing $X$ suffice to prove that $X$ is exceptional.

The first paper (\cite{hi}) appeared in 1980. In the last 10 years there were several strong (and often optimal) results on
some interpolation problems, i.e. finding a curve with prescribed degree and genus containing a large number of general points
of
$\PP^r$ or of a hyperplane of $\PP^r$ (\cite{a,aly,l3,l4,l5,l6,lv,v}) and these results are essential for the proof of Theorem
\ref{i1}.

\section{Notation}
A \emph{numerical set} or a \emph{numerical set for $\PP^n$}, $n\ge 3$, is an ordered list of integers
$\epsilon =(n;s;d_1,g_1;\dots ;d_s,g_s)$ with $n$ the dimension of the projective space, $g_i\in \NN$ and $d_i>0$ for
all $i$.
An
\emph{admissible numerical set} or an
\emph{admissible numerical set for
$\PP^n$} is a numerical set
$\epsilon = (n;s;d_1,g_1;\dots ;d_s,g_s)$ such that
$s>0$, all
$g_i\in \NN$, $d_i>0$ if $g_i=0$ and $d_i\ge \max \{2g_i-1,g_i+n\}$ if $g_i>0$. For any positive integer $k$
set $w_k(\epsilon):= \sum _{i=1}^{s} (kd_i+1-g_i)$. We say that $\epsilon$ has \emph{critical value $1$} if $g_i=0$ for all $i$
and $w_1(\epsilon) =s+d_1+\cdots +d_s\le n+1$. In all other cases the \emph{critical value} of $\epsilon$ is the first integer
$k\ge 2$ such that $w_k(\epsilon )\le \binom{n+k}{n}$. If $w_k(\epsilon)\le \binom{n+k}{n}$ for some $k\ge 2$, then
$w_{k+1}(\epsilon) \le \binom{n+k+1}{n}$ (Lemma \ref{nn1}).

An \emph{admissible generalized numerical set} or an
\emph{admissible generalized numerical set for
$\PP^n$} is an ordered list
$\epsilon =(n;s;d_1,g_1;\dots ;d_s,g_s)$ such that
$s>0$, all
$g_i$ and
$d_i$ are non-negative integers, $d_i\ge \max \{2g_i-1,g_i+n\}$ if $g_i>0$ and $d_i>0$ for at least one $i$. Thus we allow
$d_i=0$ for some $i$, but only if $g_i=0$ and we do not allow the list with $d_i=g_i=0$ for all $i$. Set $w_k(\epsilon):= 
\sum' (kd_i+1-g_i)$, where $\sum '$ means that we only sum for $i$ such that $(d_i,g_i)\ne (0,0)$. Using this definition of
$w_k(\epsilon)$ we define in the same way the critical value of $\epsilon$.

For any admissible numerical set $\epsilon =(n;s;d_1,g_1;\dots ;d_s,g_s)$ let $Z(\epsilon)$ denote the set of all smooth
curves $X\subset \PP^n$ with $s$ connected components, say $X =X_1\cup \cdots \cup X_s$ with $\deg (X_i)=d_i$,
$p_a(X_i)=g_i$ and $h^1(\Oo _{X_i}(1)) =0$. The latter condition implies that $Z(\epsilon)$ is a smooth and irreducible
subvariety of the Hilbert scheme $\mathrm{Hilb}(\PP^n)$ of $\PP^n$. Let $Z'(\epsilon)$ denote the closure of $Z(\epsilon)$ in
$\mathrm{Hilb}(\PP^n)$. We use $Z(\epsilon)$ and $Z'(\epsilon)$ for admissible generalized numerical sets just
taking
$X_i=\emptyset$ if $(d_i,g_i) = (0,0)$. We often write $Z(H;s;a_1,q_1;\dots ;a_s,q_s)$ and $Z'(H;s;a_1,q_1;\dots ;a_s,q_s)$
instead of $Z(n-1;s;a_1,q_1;\dots ;a_s,q_s)$ and $Z'(n-1;s;a_1,q_1;\dots ;a_s,q_s)$ to emphasize that their elements are
contained in $H$. 

An \emph{arrow of $\PP^m$, $m>0$} $v$ is a degree $2$ connected zero-dimensional scheme $v\subset \PP^m$. The point $v_{\red}$
is the reduction of the arrow
$v$. Note that any arrow in a projective space spans a line.

We work over an algebraically closed field with characteristic $0$.

\section{Preliminaries}
Let $H\subset \PP^n$ be a hyperplane. Let $Z\subset \PP^n$ be a closed subscheme. The residual scheme $\Res_H(Z)$ of $Z$ with
respect to $H$ is the closed subscheme of $\PP^n$ with $\Ii_Z:\Ii_H$ as its ideal sheaf. For all $t\in \ZZ$ there is an exact
sequence
\begin{equation}\label{eqhor1}
0 \to \Ii_{\Res_H(Z)}(t-1) \to \Ii_Z(t)\to \Ii_{Z\cap H,H}(t)\to 0
\end{equation}
Note that \eqref{eqhor1} implies $\chi (\Ii_{\Res_H(Z)}(t-1))+\chi (\Ii_{Z\cap H,H}(t))=\chi(\Ii _Z(t))$.
If $Z$ is a reduced curve, then $\Res_H(Z)$ is the union of the irreducible components of $Z$ not contained in $H$. If $Z$ is
a reduced curve and $\Res_H(Z)$ is transversal to $H$, then $Z\cap H$ is the union $T$ of all irreducible components of $Z$
contained in $H$ and the finite set $\Res_H(Z)\cap (H\setminus T)$. By \eqref{eqhor1} to prove that $h^1(\Ii _Z(t))  =0$ (resp.
$h^0(\Ii_Z(t)) =0$) it is sufficient to prove that $h^1(\Ii _{\Res_H(Z)}(t-1)) =h^1(H,\Ii _{Z\cap H,H}(t))=0$ (resp.
$h^0(\Ii _{\Res_H(Z)}(t-1)) =h^0(H,\Ii _{Z\cap H,H}(t))=0$). The exact sequence \eqref{eqhor1} gives $\chi (\Ii_{Z\cap H,H}(t))
=\chi (\Ii_{T,H}(t)) -
\#(\Res_H(Z)\cap (H\setminus T))$.

We use $3$ key results from \cite{aly,l5,lv}.

\begin{theorem}\label{aly1}
\cite[Corollary 1.3]{aly}). Fix integers $m\ge 3$, $g\ge 0$, $d\ge g+m$ and $x>0$. Let $S\subset \PP^m$ be a general subset
with $\# S=x$. There is a smooth and non-special curve $X\subset \PP^m$ such that $\deg (X)=d$, $p_a(X)=g$ and $S\subset X$ if
and only if
\begin{equation}\label{eqaly1}
(m+1)d +(m-3)(1-g) \ge (m-1)x
\end{equation}
unless $(d,g,m)\in  \{(5,2,3),(7,2,5)\}$. In these two exceptional cases we may/must take $x\le 9$.
\end{theorem}

\begin{theorem}\label{l5}(\cite[Theorem 1.4]{l5}) Fix integers $n\ge 3$, $g\ge 0$ and $d$ such that $(n+1)d\ge ng + n(n+1)$.
Fix a hyperplane $H\subset
\PP^n$ and a general set $S\subset  H$ such that
\begin{equation}\label{eql5}
\# S\le \min \big\{d,\frac{(n-1)^2d-(n-2)^2g -(2n^2-5n+12)}{(n-2)^2}\big\}
\end{equation}
Then there is a smooth curve $X\subset \PP^n$ of degree $d$ and genus $g$ with general moduli such that $X$ is transversal to $H$ and
$S\subseteq X\cap H$.
\end{theorem}
\begin{theorem}\label{lv4}
(\cite[Corollary 2]{lv}) Fix integers $g\ge 0$ and $d\ge g+4$. Let $H\subset \PP^4$ be a hyperplane. Let $S\subset H$ be a general subset with cardinality $d$.
Then there is a smooth and connected curve of degree $d$ and genus $g$ such that $h^1(\Oo _X(1)) =0$ and $S =X\cap H$.
\end{theorem}

\begin{remark}\label{aly10}
Fix integers $m=n-1\ge 4$, $g\ge 0$, $d\ge \max\{g+n-1,2g-1\}$ and $x>0$. Let $S\subset \PP^m$ be a general subset
with $\# S=x$. Note that $\lfloor (nd +(n-3)(1-(d+1)/2)/(n-2)\rfloor =\lfloor (nd+n-3)(2n-4)\rfloor \ge d/2$.
By Theorem \ref{aly1} there is a smooth and
non-special curve
$X\subset
\PP^m$ such that
$\deg (X)=d$,
$p_a(X)=g$ and
$S\subset X$ if $\#S \ge d/2$.
\end{remark}

\begin{remark}\label{l5.2}
Take $d, g, x$ as in Theorem \ref{aly1}. Assume $d\ge 2g-1$. We get that all $x\le \frac{d(m+5)+m-3}{2m-2}$ are allowed and in
particular we may take $x = \lceil d/2\rceil$.
\end{remark}

\begin{remark}\label{l5+}
Note that when $g=0$ in Theorem \ref{l5} we may take $\#S =d$ and that for $m=3$ in \eqref{eqaly1} the inequality
\eqref{eqaly1} is equivalent to $x\le 2d$. If we take $g=0$ and an integer $d$ such that $1\le d<m$ we take as $X$ a general
rational normal curve of a general $d$-dimensional linear subspace of $\PP^m$. Thus in this case we may take as
$x$ any integer $\le d+1$. In particular for $m=3$ and $d=2$ we may take
$x\le 3$, but not $x=4$; in this case to get a degree $2$ curve passing through $4$ general points of $\PP^3$ we will take two
disjoint lines. For the same reason in the set-up of Theorem \ref{l5} with $g=0$ if
$1\le d<n$ we may use any
$S$ such that
$\#S \le d$. 
\end{remark}

\begin{remark}\label{lv4.1}
Fix an integer $d\in \{1,2,3\}$. Since a general smooth rational curve $D\subset \PP^4$ is linearly normal in its linear span,
$D\cap H$ is formed by $d$ general points of $H$. Thus the statement of Theorem \ref{lv4} holds also for $(d,g)\in
\{(1,0),(2,0),(3,0)\}$.
\end{remark}

\begin{remark}\label{l5.1}
Take $d, g, n, x$ as in Theorem \ref{l5}, but without the assumption \eqref{eql5}. For $g\le 2$ we may take $x =d$. Assume
$g\ge 3$
and $d\ge 2g-1$. Since $g\le (d+1)/2$, \eqref{eql5} implies that we may take $x\ge d/2 -4$.
\end{remark}

\begin{remark}\label{nn3}
Let $X\subset \PP^n$ be a reduced curve such that $h^1(\Oo _X(1)) =0$. We have $h^1(\Oo _X(t))=0$ for all $t\ge 1$. Since
$\dim X=1$, we have $h^i(\Oo _X(z)) =0$ for all $i\ge 2$ and all $z\in \ZZ$. The exact sequence
$$0\to \Ii _X(t) \to \Oo _{\PP^2}(t)\to \Oo _X(t)\to 0$$gives $h^2(\Ii _X(t)) =0$ for all $t\ge 1$ and $h^i(\Ii _X(z)) =0$
for all $i\ge 3$ and $z\ge -n$. Now also assume $h^1(\Ii _X(k)) =0$ for some integer $k\ge 2$. The Castelnuovo-Mumford's
Lemma gives $h^1(\Ii _X(t)) =0$ for all $t\ge 3$.
\end{remark}

\begin{remark}\label{nn2}
Fix integers $n\ge 4$, $s>0$, $d_i$ and $g_i\ge 0$ such that we defined $Z(n;s;d_1,g_1;\dots ;g_s;d_s)$.
Let $k\ge 2$ be the critical value of the numerical set $(n;s;d_1,g_1;\dots ;g_s;d_s,g_s)$.

\quad (a) Assume $k\ge 3$. By Lemma \ref{nn1} and Remark \ref{nn3} to prove that a general $X\in Z(n;s;d_1,g_1;\dots
;g_s;d_s,g_s)$ has maximal rank it is sufficient to prove that $h^0(\Ii _X(k-1)) =0$ and $h^1(\Ii _X(k)) =0$. Since
$Z(n;s;d_1,g_1;\dots ;d_s,g_s)$ is irreducible, the semicontinuity theorem for cohomology shows that to prove that a general
$X\in Z(n;s;d_1,g_1;\dots ;d_s,g_s)$ satisfies $h^0(\Ii _X(k-1)) =0$ and $h^1(\Ii _X(k)) =0$ it is sufficient to prove the
existence of $X', X''\in Z'(n;s;d_1,g_1;\dots ;d_s,g_s)$ such that $h^0(\Ii _{X'}(k-1)) =0$ and $h^1(\Ii _{X''}(k)) =0$.

\quad (b) Assume $k=2$. By Remark \ref{nn3} to prove that a general $X=X_1\cup \cdots \cup X_s\in Z(n;s;d_1,g_1;\dots ;g_s;d_s)$
has maximal rank it is sufficient to prove that $h^1(\Ii _X(2)) =0$ and that either $h^0(\Ii _X(1)) =0$ or $h^1(\Ii _X(1)) =0$.
Since $d_i\ge g_i+n$ if $g_i>0$, we see that $h^0(\Ii _X(1)) =0$, even if $g_i=0$ for all $i$.
\end{remark}

\begin{lemma}\label{nn1}
Fix integers $n\ge 3$, $k\ge 2$, $s>0$, $d_i$ and $g_i$, $1\le i\le s$, such that $d_i>g_i\ge 0$ for all $i$ and
\begin{equation}\label{eqnn1}
\sum _{i=1}^{s} (kd_i+1-g_i)\le \binom{n+k}{n}.
\end{equation}
Then
$$\sum _{i=1}^{s} ((k+1)d_i+1-g_i)< \binom{n+k+1}{n}.$$
\end{lemma}

\begin{proof}
Assume
\begin{equation}\label{eqnn2}
\sum _{i=1}^{s} ((k+1)d_i+1-g_i)\ge \binom{n+k}{n}.
\end{equation}
Subtracting \eqref{eqnn1} from \eqref{eqnn2} we get
\begin{equation}\label{eqnn3} d_1+\cdots +d_s \ge \binom{n+k}{n-1}.
\end{equation}
Since $g_i\le d_1+1$ for all  $i$ from \eqref{eqnn1} and \eqref{eqnn3}
we get $(k-1)\binom{n+k}{n-1} \le \binom{n+k}{n}$, i.e.,
$(n+k)!(k-1)/(n-1)!k!(k+1) \le (n+k)!/n!k!$, i.e. $(k-1)/(k+1)\le 1/n$, which is false for $k\ge 2$ and $n\ge 3$.
\end{proof}

\begin{remark}\label{smoo1}
Let $Y, T\subset \PP^m$, $m\ge 3$, be smooth and irreducible connected curves such that $Y\ne T$, $X:=Y\cup T$ is nodal,
$h^1(\Oo_Y(1)) =h^1(\Oo _T(1)) =0$ and $0 < \#(Y\cap T) \le h^0(\Oo_T(1))$.  A Mayer-Vietoris
exact sequence gives
$h^1(\Oo_X(1)) =0$. The nodal curve $X$ is smoothable (\cite{hh,s}). As in \cite{be2,be4,be5}e apply several times this
observation, starting with as
$T$ a rational normal curve of a hyperplane $H$ of $\PP^m$ and then adding several lines intersecting $X\cup T$
quasi-transversally and at at most $2$ points. Any non-special curve $T'\subset H$ may be dismantled in this way.
\end{remark}

\begin{remark}\label{ma6.0}
Let $A, B\subset \PP^m$ be smooth curves meeting at a unique point, $p$. Assume that $A\cup B$ is nodal. Let $v\subset \PP^m$
be any arrow such that $v_{\red}=\{p\}$ and $v$ is not contained in the plane spanned by the tangent lines of $A$ and $B$ at
$p$. As in \cite[Ex. 2.1.1]{hh0} (case in which $A$ and $B$ are lines) one sees that $A\cup B\cup v$ a flat limit of curves $\{A\cup B_\lambda\}$ with each $B_\lambda$
projectively equivalent to $B$ and with $A\cap B_\lambda =\emptyset$. Thus $A\cup B\cup v$ is smoothable.
\end{remark}

\begin{remark}\label{k=1}
By \cite{be6} and our definition of critical value $1$, Theorem \ref{i1} is true for all $n\ge 4$ and all numerical sets
with critical value $1$.
\end{remark}

\begin{lemma}\label{k+2}
Fix a scheme $A\cup B\subset \PP^n$, $n\ge 3$ with $\dim A\le n-2$, $B$ an integral and non-degenerate curve
and $A$ not intersecting a general secant line $L$ of $B$. Then $h^0(\Ii _{A\cup B\cup L}(2)) =\max \{0,h^0(\Ii _{A\cup
B}(2))-1\}$.
\end{lemma}

\begin{proof}
To prove the lemma we may assume $h^0(\Ii _{A\cup
B}(2))>0$. Since $\deg (L\cap B)=2$ and $h^0(\Oo _L(2))=3$, we have $h^0(\Ii _{A\cup B\cup L}(2))\ge h^0(\Ii _{A\cup
B}(2))-1$. Thus it is sufficient to fix $Q\in |\Ii_{A\cup B}(2)|$ and prove that a general secant line of $B$ is not contained
in $Q$. Assume that it is contained in $Q$. Since $B$ spans $\PP^n$, a general $o\in B$ is not contained in the singular
locus of $Q$. By assumption $Q$ contains the cone $C_o(B)$ with vertex $o$ and containing $B$. Since $B$ is non-degenerate, the
Zariski tangent space of
$C_o(B)$ at $o$ is $\PP^n$. Thus $Q$ is singular at $o$, a contradiction.
\end{proof}

\begin{lemma}\label{k=2}
Fix integers $n\ge 4$, $s\ge 1$, $d_i$ and $g_i$, $1\le i\le s$ such that $d_i>0$ if $g_i=0$ and $d_i\ge g_i+n$ if $g_i>0$.
Let $Y=Y_1\cup \cdots \cup Y_s\subset \PP^n$ be a general union of $s$ general non-special curves $Y_i$ of degree $d_i$ and
genus $g_i$.
Then either $h^0(\Ii _Y(2))=0$ or $h^1(\Ii _Y(2))=0$.
\end{lemma}

\begin{proof}
Let $E\subset \PP^n$ be a general union of $s$ general smooth rational curves of degree $d_i-g_i$. Since $n\ge 4$, the lemma is
true if
$g_i=0$ for all
$i$ (it is false for
$n=3$,
$s=2$,
$g_1=g_2=0$ and
$d_1=d_2=2$). Thus it is true for $E$, i.e. either $h^0(\Ii _E(2))=0$ or $h^1(\Ii _E(2))=0$.
We may obtain $Y$ as a flat deformation of the union $E\cup T$ of $E$ and $g_1+\dots +g_s$ disjoint lines, exactly $g_i$ of
them secant to
$E$. By the semicontinuity theorem for cohomology it is sufficient to prove that either $h^0(\Ii_{E\cup T}(2)) =0$
or $h^0(\Ii _{E\cup T}(2))=0$. The lemma is obvious if $h^0(\Ii_E(2))=0$. If $h^1(\Ii_E(2))=0$ we apply $g_1+\cdots+g_s$ times
Lemma
\ref{k+2}.
\end{proof}

\begin{lemma}\label{rnn0}
Take a vector space $W\subseteq H^0(\Oo _{\PP^n}(2))$ and an integral curve $T\subset \PP^n$ contained in the base locus of
$W$. For any scheme
$E\subset
\PP^n$ set $W(-E):= H^0(\Ii _E(2))\cap W$. For a general
$o\in T$ and a general line $L\subset \PP^n$ such that $o\in L$ we have $\dim W(-L) = \max \{0,\dim W-2\}$, unless all $Q\in
|W|$ are cones with vertex containing $T$.
\end{lemma}

\begin{proof}
Since $W(-o) =W$, we have $\dim W(-L)\ge \dim W-2$. Since $L$ contains a general point of $\PP^n$, we have $\dim W(-L) \ge \max
\{0,\dim W-1\}$. Thus we may assume $\dim W\ge 2$. Foix a general $p\in \PP^n$. Thus $\dim W(-p) =\dim W-1$. Fix $Q\in
W(p-p)$. We are done, unless every line containing $p$ and intersecting $T$ is contained in $Q$, i.e. unless $Q$ contains the
cone $C_p(T)$ with vertex $p$ and base $T$. Take $p'\in Q$ near $p$. We still have $\dim W(-p') =\dim W(-p)$ by semicontinuity.
Thus $\dim W(-L) = \max \{0,\dim W-2\}$, unless $Q$ is a cone with vertex containing $T$. Since we may take as $Q$ any element
of $|W|$ containing a general $p\in \PP^n$, we conclude.
\end{proof}

\begin{lemma}\label{rplan1}
Let $W\subseteq H^0(\Oo _{\PP^n}(k))$ be a linear subspace. Fix a general plane $E\subset \PP^n$ and a general $o\in E$. Let
$(2o,E)$ denote the first infinitesimal neighborhood $(2o,E)$ of $o$ in $E$. We have $\dim W(-(2o,E)) =\max \{0,\dim W-3\}$,
unless the rational map
$\phi$  induced by $|W|$ sends $\PP^n$ onto a curve.
\end{lemma}
\begin{proof}
We have $\deg ((2o,E)) =3$ and hence $\dim W(-(2o,E)) \ge \max \{0,\dim W-3\}$. Thus we may assume $\dim W\ge 3$. Fix a
general $o\in \PP^n$. The integer $\dim W -\dim W(-2o)$ is the rank of the differential of $\phi$ at $o$. This rank is $\le 1$
if
$\dim W(-(2o,E)) \ge \dim W-2$ for a general plane $E$ containing $o$. In characteristic $0$ the rank of the differential of
$\phi$ at a general point is the dimension of the image.
\end{proof}

\begin{remark}\label{rplan2}
Let $W\subseteq H^0(\Oo _{\PP^n}(k))$ be a linear subspace. Fix a general line $L\subset \PP^n$ and a general $o\in E$.   Let
$(2o,L)$ denote the first infinitesimal neighborhood $(2o,L)$ of $o$ in $L$. The proof of Lemma \ref{rplan1} gives
$\dim W(-(2o,L)) =\max \{0,\dim W-2\}$
\end{remark}

\section{Space curves}\label{S3}
To prove the case $n=4$ of Theorem \ref{i1} we need several results for curves  a hyperplane of $\PP^4$, i.e. several results
for curves in $\PP^3$. We will use in an essential way \cite{b}, but we also need reducible connected curves which may be
chopped into connected subcurves with prescribed degrees. As in \cite{b,be6} we use connected curves of arithmetic genus $0$
with lines as irreducible components. We also write several exceptional cases to the extension of Theorem \ref{i1} to the case
$n=3$ (Remark \ref{a3} and Lemmas \ref{3ma1} and \ref{3ma2}), but we do not claim to have the full list of the exceptional
cases. We say that
$(s;d_1,g_1;\dots ;d_s,g_g)$ is \emph{admissible for $\PP^3$} if $d_i>0$ for all $i$ and $d_i\ge g_i+3$ if $g_i>0$.

A degree
$d$
\emph{tree}
$T\subset
\PP^r$,
$r\ge 3$, is a connected nodal curve of degree
$d$ and arithmetic genus
$0$ whose irreducible components are lines. A \emph{forest} in $\PP^r$ is a union of finitely many disjoint trees. For all
positive integers $s$,
$d$ and $d_i$,
$1\le i\le s$, let $T(r;d)$ denote the set of all degree $d$ trees of $\PP^r$ and $T(r;s;d_1,\dots ,d_s)$ the set of all
forests in $\PP^r$ with $s$ connected components of degree $d_1,\dots ,d_s$. The sets $T(r;d)$ and $T(r;s;d_1,\dots,d_s)$ are
smooth quasi-projective varieties. If $d\ge 4$ the set $T(r;d)$ is not irreducible, but we may describe its irreducible
components, which are also its connected components,in the following way. Fix an integer $d\ge 2$ and let $T\subset \PP^r$ be a
degree $d$ tree. It is easy to see the existence of an ordering $L_1,\dots ,L_d$ of the irreducible components of $T$ such that
for all $i=1,\dots ,d$ the curve $\cup _{1\le j\le i} L_j$ is connected. We  say that any such ordering is admissible. 
Let $T(r;d)'$ be the set of pairs $(T,\leq)$, where $T\in T(r;d)$ and $\leq $ is an admissible ordering of $T$. Fix
one such ordering. Since
$T$ is nodal, connected and with arithmetic genus $0$, for each
$i\in \{2,\dots,d\}$ there is a unique $\tau(i)\in \{1,\dots ,i-1\}$ such that $L_i\cap L_{\tau(i)}\ne \emptyset$. Thus the
admissible ordering of $T$ induces a map $\tau: \{2,\dots ,d\}\to \{1,\dots ,d-1\}$ such that $\tau(i)<i$ for all $i$.
As in \cite{b,be6} we say that $\tau$ is the \emph{type} of $(T,\leq)$. Let $\tau: \{2,\dots ,d\}\to \{1,\dots ,d-1\}$ such
that $\tau(i)<i$ for all $i$. It is obvious how to construct a pair $(T,\leq)\in T(r;d)'$ such that $\tau$ is the type
of $(T,\leq)$. For any function $\tau: \{2,\dots ,d\}\to \{1,\dots ,d-1\}$ such that $\tau(i)<i$ for all $i$ let
$T(r;d;\tau)$ denote the set of all $T\in T(r;d)$ with an admissible ordering $\leq$ such that $(T,\leq)$ has type $\tau$. It
is easy to see that the sets $T(r;d;\tau)$ are the irreducible components of $T(r;d)$, except that we may have
$T(r;d;\tau)=T(r;d;\tau ')$ for some $\tau '\ne \tau$. Note that either $T(r;d;\tau)=T(r;d;\tau ')$ or
$T(r;d;\tau)\cap T(r;d;\tau ') =\emptyset$. Thus the sets $T(r;d;\tau)$ are the connected components of $T(r;d)$, too.
A degree $d$ tree $T\subset \PP^r$ is called a \emph{bamboo} if either $d=1$ or $d\ge 2$ and there is an admissible ordering
$\leq$ of $T$ with type $\tau(i)=i-1$ for all $i$. Thus a degree $d\ge 2$ tree is a bamboo if and only if there is an ordering
$L_1,\dots ,L_d$ of the irreducible components of $T$ such that $L_i\cap L_j\ne \emptyset$ if and only if $|i-j|\le 1$. If
$d\le 3$ every degree $d$ tree is a bamboo. Let $B(r;d)$ denote the set of all degree $d$ bamboos contained in $\PP^r$.
For all integers $r\ge 3$, $s>0$ and $d_i>0$, $1\le i \le s$, let $B(r;s;d_1,\dots ,d_s)$ denote the subsets of
$T(r;s;d_1,\dots ;d_s)$ formed by all forests whose connected components are bamboos.

\begin{remark}\label{t1}
Fix a hyperplane $H\subset \PP^r$, $r\ge 3$, a positive integer $d$, a type $\tau: \{2,\dots ,d\}\to \{1,\dots ,d-1\}$ for
degree $d$ trees and a finite set
$S\subset H$ such that
$\#S =d$. Since any two points of $\PP^r$ are collinear, it is easy to show the existence of $T\in T(r;d;\tau)$ such that
$T\cap H=S$. Note that any such $T$ intersects transversally $H$.
\end{remark}

\begin{remark}\label{t2}
Fix a hyperplane $H\subset \PP^r$, $r\ge 3$, a positive integer $d$, a type $\tau: \{2,\dots ,d\}\to \{1,\dots ,d-1\}$ for
degree $d$ trees and a finite set
$A\subset H$ such that $\#A\le \lfloor d/2\rfloor$. Since any two points of $\PP^r$ are collinear, it is easy to show the existence of $T\in
B(r;d)$ such that
$A = \mathrm{Sing}(T)$ and no irreducible component of $T$ is contained in $H$.
\end{remark}

Remark \ref{t2} is not true for many types $\tau$ of degree $d\ge 4$ trees. We just give an example.

\begin{example}
Consider the map $\tau: \{2,\dots ,d\}\to \{1,\dots ,d-1\}$ with $\tau(i)=1$ for all $i$. We call trees with this type
\emph{spreading} trees. Assume $d\ge 4$. Fix a hyperplane $H\subset \PP^r$ and let $T=L_1\cup \dots \cup L_d$ be a spreading
tree. Since $\mathrm{Sing}(T)\subset L_1$, any hyperplane containing $2$ singular points of $T$ contains $L_1$.
\end{example}
We say that $\emptyset$ is the type of a degree $1$ tree. For all integers $r\ge 3$, positive integers $d_i$, $1\le i\le s$,
and types $\tau _i$ for degree $d_i$ trees, $1\le i\le s$ let $T(r;s;d_1,\tau_1,\dots ,d_s,\tau _s)$ denote the set of all
$T_1\cup \cdots \cup T_s\in T(r;s;d_1,\dots ,d_s)$ with $T_i$ of type $\tau_i$. $T(r;s;d_1,\tau_1,\dots ,d_s,\tau _s)$ is an
irreducible and smooth quasi-projective variety.

Let $T$ be a degree $d$ tree. If $d=1$ we will say that $T$ is the \emph{final line} of $T$. If $d\ge 2$, a line of $T$ is said to be a \emph{final line} if it meets only another irreducible component of $T$. $T$ had at least $2$ final lines and it has exactly $2$ final lines if and only if it is a bamboo.

Let $Q\subset \PP^3$ be a smooth quadric. We have $\mathrm{Pic}(Q)\cong \ZZ^{\oplus 2}$ and we call $\Oo_Q(1,0)$ and
$\Oo_Q(0,1)$ the free generators of $\mathrm{Pic}(Q)$ whose associated linear systems are the two rulings of $Q$. Fix
$T\in B(3;s;d_1,\dots ,d_s)$,
$s\ge 2$, an an integer
$k$ such that
$0<k<s$. We will say that
$T$ has
\emph{$k$ good secants} or that \emph{$Q$ contains $k$ good secants of $T$} if there are $k$ disjoint lines $L_1,\dots
,L_k\subset Q$ such that
$T\cup L_1\cup
\cdots
\cup L_k\in B(3;s-k,a_1,\dots ,a_{s-k})$ for some positive integers $a_1,\dots ,a_{s-k}$.

We recall the following result, proved in a preliminary version of \cite{be01}; the case in which all types are bamboos is
stated in \cite[Claim at p. 592]{be6}.

\begin{proposition}\label{be1}
Fix positive integers $a$, $b$, $s$, and $d_i$, $1\le i\le s$. Take types $\tau_i$ for degree $d_i$ trees of $\PP^3$.
Take a general $T\in T(3;s;d_1,\tau_1,\dots ,d_s,\tau _s)$. The set $T\cap Q$ has maximal rank with respect to the line bundle
$\Oo_Q(a,b)$, unless $(a,b,s,d_1) =(1,1,1,2)$.
\end{proposition}

Let $H\subset \PP^3$ be a plane and let $Q\subset \PP^3$ be a smooth quadric surface. In some of the proofs we will add some
restrictions to the choice of $H$ or $Q$. For any closed scheme $Z\subset \PP^3$ the residual scheme $\Res_Q(Z)$ of $Z$ with
respect to $Q$ is the closed subscheme of $\PP^3$ with $\Ii_Z:\Ii_Q$ as its ideal sheaf. For every $t\in \ZZ$ the following
exact sequence
$$0 \to \Ii_{\Res_Q(Z)}(t-2)\to \Ii_Z(t)\to \Ii_{Z\cap Q,Q}(t)\to 0$$ will be called the residual exact sequence of $Q$.

\begin{remark}\label{a3}
Let $X\subset \PP^3$ be a general union of $s$ smooth rational curves of degree $d_1\ge \dots \ge d_s>0$. By \cite{b} $X$ has maximal rank, except in the following cases:
\begin{enumerate}
\item $s=2$, $d_1=d_2=2$;
\item $s=2$, $d_1=4$, $d_2=2$;
\item $s=3$, $d_1=d_2=d_3=2$;
\item $s=3$, $d_1=4$, $d_1=d_2=2$;
\item $s=4$, $d_1=d_2=d_3=d_4=2$.
\end{enumerate}
In case (1) we have $h^0(\Ii _X(2)) =h^1(\Ii _X(2)) =1$. In cases (2) and (3) we have $h^0(\Ii _X(3)) >0$, because $X$ is contained in a reducible cubic surface. For a general $X$ we also have $h^0(\Ii _X(3)) =1$ and hence $h^1(\Ii _X(3)) =1$ in case (2) and $h^1(\Ii _X(3)) =2$ in case (3). In cases (4) and (5) $X$ is contained in a reducible quartic surface.
Using the residual scheme with respect to a plane it is easy to see in both cases that $h^0(\Ii _X(4)) =1$ and hence $h^1(\Ii _X(4)) =1$ in case (4) and $h^1(\Ii _X(4)) =2$ in case (5). For all other integers $t$ either $h^0(\Ii _X(t)) =0$ or $h^1(\Ii _X(t))=0$.
\end{remark}

\begin{remark}\label{lv2}
Fix integers $d$ and $g$ in the Brill-Noether range for $\PP^3$, i.e. take $(d,g)\in \NN^2$ such that $4d \ge 3g+12$ and let
$Z(3;1;d,g)$ denote the irreducible component of $\mathrm{Hilb}(\PP^3)$ containing the curves with general moduli.
Fix a general $X\in Z(3;1;d,g)$. By a theorem of E. Larson (\cite[Theorem 1.4]{l4}, quoted also in  \cite[Theorem 1.4]{v}) we
have $h^1(N_X(-2)) =0$, except if $(d,g)$ is one of the following ones:
$$(4,1),\ (5,2),\ (6,2),\ (6,4), \ (7,5),\ (8,6).$$
In the non-special range, i.e. for $d\ge g+3$, there remain only the pairs $(4,1)$, $(5,2)$ and
$(6,2)$. This is obviously true also if $(d,g)=(1,0)$, but it is not true if $(d,g) =(2,0)$. In this case for any curvilinear
scheme
$Z\subset
\PP^3$ such that $\deg (Z)=3$ and $Z$ is not contained in a line there is a smooth conic $D\supset Z$.

Fix integers $g\ge 0$ and $d\ge g+3$. Let $S\subset Q$
be a general subset with cardinality $2d$. If $(d,g)\notin \{(4,1),(5,2),(6,2)\}$ by the quoted theorem of E. Larson there is
$X\in Z(3;1;d,g)$ such that $X\cap Q=S$ (use \cite[Theorem 1.5]{pe}). 

Now assume $(d,g)=(4,1)$. Since any $X\in Z(3;1;4,1)$ is the complete intersection of $2$
quadric surfaces, a general  complete intersection $S'$ of two general element of
$|\Oo_Q(2,2)|$ there is $X'\in Z(3;1;4,1)$ such that $S' = X'\cap Q$. Thus for a general $A\subset Q$ with $\#A = 7$ there is
$X\in Z(3;4,1)$ intersecting transversally $Q$ and with $A\subset X\cap Q$; moreover $X\cap Q$ is the complete intereection of
two any two elements of $|\Ii _{A,Q}(2,2)|$.

Now assume $(d,g) =(5,2)$. A general element of $Z(3;1;5,2)$ is an element of $|\Oo _{Q'}(2,3)|$ for some smooth quadric $Q'$.
Thus for a general $A\subset Q$ such that $\#A=8$ there is $X\in Z(3;1;5,2)$ intersecting transversally $Q$ and with $A\subset
X\cap Q$.

Now assume $(d,g) =(6,2)$. For a general $S\subset \PP^3$ such that $\#S =9$ there is $X\in Z(3;1;6,2)$ containing $S$ and $9$
is the maximal positive integer with this property (\cite[Theorem 1.1]{v}). There are infinitely many such curves $X$'s but
for a general $S$ we are sure there is a general one and in particular we may find $X$ containing $S$ and with the Hilbert
function of a general element of $Z(3;1;6,2)$. By \cite{be4} we may find $X$ such that $h^0(\Ii_X(2)) =1$. Fix a general
$A\subset Q$ such that
$\#A = 9$. Since any
$9$ point of $\PP^3$ are contained in a quadric surface, for a general $Q$, $A$ is a general subset of $\PP^3$ with
cardinality $9$. By \cite[Theorem 1.1]{v} there is $X\in Z(3;1;6,2)$ such that $A\subset X$ and $h^0(\Ii _X(2))=0$. The latter
condition implies $X\nsubseteq Q$, i.e. $\dim X\cap Q =0$.
\end{remark}

For all $m\in \NN$ define the integers $a(m)$ and $q(m)$ by the following relations:
\begin{equation}\label{eqio1}
mr(m)+1+q(m) =\binom{m+3}{3}, \ 0\le q(m) \le m
\end{equation}
The integers $q(m)$ only depend on the congruence class of $m$ modulo $6$. For all $k\in \NN$ we have
$r(6k+1)=6k^2+8k+3$, $q(6k+1)=0$, $r(6k+2)=6k^2+10k+4$, $q(6k+2)=3k+1$, $r(6k+3)=6k^2+12k+6$,
$q(6k+3)=2k+1$, $r(6k+4)=6k+4)=6k^2+14k+8$, $q(6k+4)=3k+2$, $r(6k+5)=6k^2+16k+11$, $q(6k+5)=0$, $r(6k+6)=6k^2+18k+13$,
$q(6k+6)=5k+5$.

Let $X\subset \PP^3$ be a curve whose connected components are bamboos. As in \cite{b} we say that
$Q$ has $x$ \emph{good secants to $X$} (with respect to the fixed smooth quadric $Q$), $x$ a positive integer, if there are $x$
disjoint lines
$L_1,\dots ,L_x$, i.e.
$x$ elements of one of the $2$ rulings of $Q$, such that $X\cup L_1\cup \cdots \cup L_x$ has bamboos as connected components
and $x$ connected components less than $X$.

We proved the following assertions $B(m)$, $m$ a positive integer, depending on the congruence class of $m$ modulo
$6$ (\cite[\S 5]{b}):

\quad {\bf Assertion} $B(m)$ if $m\equiv 2,3,4,6\pmod{6}$: There exists $(Z,Q)$ such that
\begin{enumerate}
\item $Z\subset \PP^3$ is the union of $q(m)+1$ disjoint bamboos, $\deg (Z)=r(m)$ and $h^i(\Ii_Z(m))=0$, $i=0,1$;
\item $Q$ is a smooth quadric containing $\Sing(Z)$, $\dim Z\cap Q=0$, and $Q$ has $q(m)$ disjoint good secants to $Z$.
\end{enumerate}

\quad {\bf Assertion} $B(m)$ if $m\equiv 1,5\pmod{6}$: There exists a pair $(Z,Q)$ such that
\begin{enumerate}
\item $Z\subset \PP^3$ is the union of $m+1$ disjoint bamboos, $\deg (Z)=r(m)-1$, and $h^i(\Ii_Z(m))=0$, $i=0,1$;
\item $Q$ is a smooth quadric containing $\Sing(Z)$, $\dim Z\cap Q=0$, and $Q$ has $m$ disjoint good secants to $Z$.
\end{enumerate}

Let $Y\subset\PP^3$ be a tree. Fix $S\subseteq \Sing(Y)$. Since $Y$ has only ordinary nodes as singularities
by our definition of tree, $Y\setminus S$ has at most $\#S +1$ connected components. We will say that $S$ is a \emph{good set
of nodes} if $Y\setminus S$ has $\#S +1$ connected components. Set $d:= \deg (Y)$. Fix positive integers $s$ and $d_i$, $1\le
i\le s$, such that $d_1+\cdots +d_s=d$. We say that $S$ is a \emph{good set of
nodes for $(s;d_1,\dots ,d_s)$} if $\#S =s-1$ and  $Y\setminus S$ has $s$ connected components whose closures $Y_1,\dots ,Y_s$
have degrees $d_1,\dots ,d_s$. Note that each $Y_i$ is a tree.

\begin{lemma}\label{io1}
Fix an integer $m\ge 4$ such that $m\equiv 0,2,4,5\pmod{6}$ and fix positive integers $s$, $d_i$, $1\le i\le s$, such that
$d_1+\cdots +d_s =r(m)$. Then there exists a pair $(Y,S)$, where $Y$ is a degree $r(m)$ tree, $S\subset \mathrm{Sing}(Y)$ is a
good set of nodes for $(s;d_1,\dots ,d_s)$, $h^1(\Ii_Y(m)) =0$ and $h^0(\Ii _Y(m)) =q(m)$.
\end{lemma}

\begin{proof}
Take $(Z,Q)$ satisfying $B(m-2)$ with the $q(m-2)$ good secants $L_i$, $1\le i \le q(m-2)$, with $L_i\in |\Oo_Q(1,0)|$ for all
$i$. Note that $r(m)-r(m-2)\ge q(m)$. Set $Y:= Z\cup L_1\cup \cdots \cup L_{r(m)-r(m-2)}$, where the lines $L_i\in
|\Oo_Q(1,0)|$,
$q(m-2)<i<r(m)-r(m-2)$, are chosen in the following way. Each $L_i$, $i>q(m-2)$, contains exactly one point of $Z\cap Q$, but
we will specify later  the allowable choices. For any allowable choice we get a tree $Y$ such that $h^1(\Ii_Y(m)) =0$  and
hence
$h^0(\Ii _Y(m)) =q(m)+1$ (\cite[Proofs in \S 4]{b}). We need to chose the lines $L_i$ so that $Y$ has a good set of
nodes for $(s;d_1,\dots ,d_s)$. Set
$E:= Z\cup L_1\cup
\cdots
\cup L_{q(m-2)}$. By the definition of
$B(m-2)$ the curve $E$ is a bamboo and hence any subset of $\Sing(E)$ is a good set of nodes for $E$. Call $e$ the number of
$i\in
\{1,\dots ,s\}$ such that $d_i=1$. First assume $e\ge r(m)-r(m-2)-q(m-2)$. In this case we take as points of $S$ points
$L_i\cap Z$ for $i>q(n)$, plus some singular point of $E$. Now assume $e< r(m)-r(m-2)-q(m-2)$, say
$d_i=1$ for $r(m)-r(m-2)-q(m)-e < i \le r(m)-r(m-2)-q(m-2)$. Among the points of a good set for $(s;d_1,\dots ,d_s)$ we take
the points $L_i\cap Z$, $r(m)-r(m-2)-q(m)-e < i \le r(m)-r(m-2)-q(m-2)$. To complete the good set it is sufficient to find a
good set $S'\subseteq \mathrm{Sing}(E)$ for $(s-e;d_1,\dots ,d_{s-e})$. We take a good ordering, say $D_1,\dots ,D_x$,
$x=\deg(E)$,
of the irreducible components of the bamboo $E$ and mark as green the points of $Z\cap (Q\setminus L_1\cup \cdots
L_{q(m-2)})$. We label the green points with an order compatible with the ordering of the components of $E$. Thus $D_i$
has $0$ green points if either $D_i =L_j$ for some $j$ or it meets two lines $L_j$ and $L_h$, one green points if it meets
exactly one line $L_j$ and two green points if it meets no line $L_j$. We take as $L_{q(m-2)+1}$ the only element
of $|\Oo_Q(1,0)|$ containing the first green point, $L_{q(m-2)+2}$ the only element of $|\Oo_Q(1,0)|$
containing the secong green point, and so on until, say after $L_{q(m-2)+1},\dots ,L_{q(m-2)+1}$, the part of the bamboo $E$
containing $L_{q(m-2)+1},\dots ,L_{q(m-2)+1}$ consists of $d_1$ lines. We take a mark at the next singular point of $E$ and
continue with $d_2, \dots ,d_s$ to get a good set for $(s;d_1,\dots ;d_s)$.
\end{proof}

\begin{lemma}\label{io2}
Fix an integer $m\ge 7$ such that $m\equiv 1,3\pmod{6}$ and fix positive integers $s$, $d_i$, $1\le i\le s$, such  that
$d_1+\cdots +d_s =r(m)-1$. Then there exists a pair $(Y,S)$, where $Y$ is a degree $r(m)-1$ tree, $S\subset \mathrm{Sing}(Y)$
is a good set of nodes for $(s;d_1,\dots ,d_s)$, $h^1(\Ii_Y(m)) =0$ and $h^0(\Ii _Y(m)) =q(m)+m$.
\end{lemma}

\begin{proof}
We take $(Z,Q)$ satisfying $B(m-4)$ and add  $r(m-2)-r(m-4)$ disjoint lines of $Q$, $q(m-4)$ the good secants of $Q$ for
$Z$ and the other ones meeting $Z$ at a unique point. Then we deform the tree $Y'$ obtained in this way to a general
tree $Y''$  with the same type and transversal to $Q$. Then inside $Q$ we add  $r(m)-r(m-2)-1$ disjoint lines, each of them
meeting $Y''$ at a unique point. We get a tree $Y$.
As in \cite{b} we first prove that $h^1(\Ii_{Y'}(m-2))=0$ and then $h^1(\Ii_Y(m)) =0$ and $h^0(\Ii _Y(m)) =q(m)+m$. In the
first (resp. second)  step we attach the lines to $Z$ (resp. $Y''$) in such a way that the tree $Y$ has a good sets of nodes
for $(s;d_1,\dots ,d_s)$.
\end{proof}

\begin{lemma}\label{3ma1}
Fix integers $s\ge 1$, $d_i>0$ and $g_i\ge 0$, $1\le i\le s$, such that $d_i>0$ for all $i$ and $d_i\ge g_i+3$ if $g_i>0$. For
all
$i<j$ assume
$g_i\ge g_j$ and $d_i\ge d_j$ if $g_i=g_j$. Take a general
$Y\in Z(3;s;d_1,g_1;\dots ;d_s,g_s)$. Set $e:= 20 -\sum _{i=1}^{s} (3d_i+1-g_i)$. We have $h^0(\Ii _Y(4)) = \max \{0,e\}$
and $h^1(\Ii _Y(3)) =\max \{0,-e\}$ if and only if $(s;d_1,g_1;\dots ;d_s,g_s)$ is not in one of these exceptional cases:

\begin{enumerate}
\item $g_i=0$ for all $i$,  and $(s;d_1,\dots ,d_s)\in \{(2;4,2),(3;2,2,2)\}$;
\item $s=2$, $g_1 =3$, $d_1=6$, $d_2=1$, $g_2=0$;
\item $s=2$, $1\le g_1\le 2$, $d_1=5$, $g_2=0$, $d_2=1$;
\item $s=3$, $d_1=5$, $g_1=2$, $g_2=g_3=0$, $d_1=d_2=1$;
\item $s=2$, $d_1=4$, $g_1=1$, $g_2=0$ and $d_2=1$;
\item $s=3$, $d_1=4$, $g_1=1$, $g_2=g_3=0$ and $d_2=d_3=1$.
\end{enumerate}
\end{lemma}

\begin{proof}
We have $\binom{6}{3} =20$.

The case $s=1$ of both assertions is a very particular case of the maximal rank conjecture for non-special space curves
(\cite{be4}). Case (2) of both assertions follows from
\cite{b}. Thus from now on we assume $s\ge 2$ and $g_1>0$. Thus $d_1\ge g_1+3$.

In some of the steps we will give other restrictions on $H$ or $Q$.

In steps (a), (b), ({c}), (d) and (e) we assume $\sum _i (3d_i+1-g_i) \le 20$ and see if $h^1(\Ii _Y(3))=0$.

\quad (a) Assume $s=2$, $g_1 =3$, $d_1=6$, $d_2=1$, $g_2=0$. Take as $Y_1$ a general smooth curve of degree $6$ and genus $3$
and $Y_2$ a general line of $H$. The scheme $Y\cap H$ is the union of a line $Y_2$ and $6$ points of $H\setminus Y_1$
contained in no conic. Thus $h^i(H,\Ii _{Y\cap H,H}(3))=0$, $i=0,1$. We have $h^i(\Ii _{Y_2}(2)) =0$, $i=0,1$. Use the
residual exact sequence of $H$.

\quad (b) Assume $s=2$, $1\le g_1\le 2$, $d_1=5$, $g_2=0$, $1\le d_2\le 2$. We take a smooth plane curve $Y_2\subset H$. We
take as $Y_1$ a general smooth curve of degree $5$ and genus $g_1$. We have $h^1(\Ii _{Y_1}(2)) =0$ (\cite{be4}). Since
$Y_1\cap H$ is the union of
$5$ points of
$H\setminus Y_2$, no $3$ of them collinear, we have $h^1(H,\Ii _{Y\cap H,H}(3)) =0$. Use the residual exact sequence of $H$.

\quad ({c}) Assume  $s=3$, $d_1=5$, $g_1=2$, $g_2=g_3=0$, $d_1=d_2=1$. Let $Y_1\subset \PP^3$ be a general smooth curve of
degree $5$ and genus $2$. Let
$T\subset H$ be a reducible conic whose singular point is a general point, $o$, of $H$. Since $h^1(\Ii _{Y_1}(2)) =0$
(\cite{be4}), we have
$h^0(\Ii _{Y_1}(2)) =1$. Since $o$ is a general point of a plane $H$ and we may take $H$ general after fixing $Y_1$, we have
$h^i(\Ii_{Y_1\cup
\{o\}}(2))=0$,
$i=0,1$. Let
$v\subset
\PP^3$ be a general arrow of
$\PP^3$ with
$o$ as its reduction. The curve
$T\cup v$ is a flat limit of a family of pairwise disjoint lines (\cite{hh0}). Set $X:= Y_1\cup T\cup v\in
Z'(3;3;5,2;1,1;1,1)$ (Remark \ref{ma6.0}). We have $X\cap H =T\cup (Y_1\cap H)$ and hence $h^i(H,\Ii_{X\cap H,H}(3))=0$,
$i=0,1$. Use that
$\Res _H(X) =Y_1\cup \{o\}$ and the residual exact sequence of $H$.

\quad (d) Assume $s=3$, $d_1=4$, $g_1=1$, $g_2=g_3=0$ and $d_2=d_3=1$. Let
$Y_1\subset
\PP^3$ be a general smooth curve of degree $4$ and genus $1$. We have $h^i(\Ii _{Y_1}(1)) =0$, $i=0,1$, and $Y_1\cap Q$
are $8$ points of $Q$ which are the complete intersection of $2$ general elements of $|\Oo_Q(2,2)|$. Thus $h^0(Q,\Ii
_{Y_1\cap  Q,Q}(3,1))=0$ and hence $h^1(Q,\Ii _{Y_1\cap Q,Q}(3,1))=0$. Take $Y_2,Y_3\in |\Oo_Q(0,1)|$ such that $Y_2\ne Y_3$
and $Y_2\cup Y_3$ contains no point of $Y_1\cap Q$. Set $Y:= Y_1\cup Y_2\cup Y_3$. We have $Y\cap Q = Y_2\cup
Y_3\cup (Y_1\cap Q)$
and hence $h^0(Q,\Ii _{Y\cap Q,Q}(3,3)) = h^0(Q,\Ii _{Y_1\cap Q,Q}(3,1)) =0$. Thus $h^1(Q,\Ii _{Y\cap Q,Q}(3,3)) =0$.
The residual exact sequence of $Y$ gives $h^i(\Ii _Y(3)) =0$, $i=0,1$.

\quad (e) Assume $s=2$, $d_1=4$, $g_1=1$, $g_2=0$ and $d_2$. This curve $Y$ is the union of $2$ of the connected components
of the curve $Y'$ described in step (d). Since $h^1(\Ii _{Y'}(3)) =0$, we have $h^1(\Ii _Y(3)) =0$.

\quad (f) Now we prove the assertion concerning $h^0(\Ii_Y(3))$. If $e=0$, we are in the case $(2;5,2;1,0;1,0)$ considered in
step ({c}). See Example \ref{3ma1.1} for the case $(3;2;4,1;2,0)$. The other exceptional cases have $g_i=0$ for all $i$
(\cite{b} or Remark \ref{a3}). Now assume
$g_1>0$ and
$e\le -2$. If $(s;d_1-1,g_1-1;d_2,g_2;\dots ;d_s,g_s)$ (which has $e':= e+2$) is not an exceptional case, then adding any
secant line of the component curve $Y'_1$ of the curve $Y'$ for $(s;d_1-1,g_1-1;d_2,g_2;\dots ;d_s,g_s)$ shows that
$(s;d_1-1,g_1-1;d_2,g_2;\dots ;d_s,g_s)$ is not an exceptional case. Since a general $p\in \PP^3$ is contained in a secant
line of $Y'_1$, we see that even if $Y'$ is exceptional, but $h^0(\Ii _{Y'}(3)) =1$, then $(s;d_1-1,g_1-1;d_2,g_2;\dots
;d_s,g_s)$. 
\end{proof}

\begin{example}\label{3ma1.1}
Take a general $Y =Y_1\cup Y_2\in Z(3;2;4,1;2,0)$. We claim that 
$h^1(\Ii _Y(3)) =1$ and $h^0(\Ii _Y(3))=2$ and hence that $Y$
has not maximal rank. Since $h^0(\Oo _Y(3)) =19$, to prove the claim it is sufficient to prove that $h^0(\Ii _Y(3))=2$.
Let $H$ be the plane spanned by $Y_2$. Since $Y_1$ is the complete intersection of two quadric surfaces, it is obvious that
$h^0(\Ii _Y(3))\ge 2$ and that equality holds if and only if every cubic surface containing $Y$ has $H$ as one of its
irreducible components. Since $Y_1$ is general, $Y_1\cap H$ is formed by $4$ non-collinear points. Thus $h^0(H,\Ii _{Y\cap
H,H}(3)) =h^0(H,\Ii _{Y_1\cap H,H}(1)) =0$. Hence $H$ is in the base locus of $|\Ii_Y(3)|$.
\end{example}

\begin{example}\label{3ma2.0}
Let $Y=Y_1\cup Y_2\subset \PP^3$ be a general union of a smooth curve $Y_1$ of degree $5$ and genus $2$ and an elliptic curve
$Y_2$ of degree $4$. We claim  that
$$ h^0(\Ii _Y(4)) = h^1(\Ii_Y(4))=2$$
and hence $Y$ as not maximal rank. We have $\chi (\Oo _Y(4)) =35$ and hence $h^0(\Ii _Y(4)) =h^1(\Ii_Y(4))$. Since
$h^0(\Oo_{Y_1}(2)) = 9<h^0(\Oo_{\PP^3}(2))$,
$Y_1$ is contained in a quadric $Q$, obviously irreducible. Since $\deg (Y_1)>4$, Bezout give $\{Q\} = |\Ii _{Y_1}(2)|$. Since
$h^0(\Ii_{Y_2}(2))=2$, using reducible quartics we get $h^0(\Ii _Y(4)) \ge 2$. To conclude the proof of the claim it is
sufficient to
prove that each $W\in |\Ii _Y(4)|$ has $Q$ as one of its irreducible components, i.e. it is sufficient to prove that $h^0(Q,\Ii
_{Y_1\cup (Y_2\cap Q)}(4)) =0$. By the semicontinuity theorem for cohomology it is sufficient to prove it when $Y_1$ is general
and in particular we may assume that $Q$ is a smooth quadric surface. Up to renaming the two rulings of $Q$ we have
$Y_1\in |\Oo_Q(3,2)|$. Thus it is sufficient to prove that $h^0(Q,\Ii _{Y_2\cap Q}(1,2)) =0$ for a general $Y_2$, which is
obvious,
because we may fix a smooth $D\in |\Oo_Q(2,2)|$ (thus $D$ is an elliptic curve and $h^0(Q,\Ii_D(1,2)) =h^0(Q,\Oo_Q(-1,0))=0$)
and take as
$Y_2\cap Q$
$8$ points of $D$.
\end{example}

\begin{example}\label{3ma2.1}
Let $Y=Y_1\cup Y_2\subset \PP^3$ be a general union of two smooth curves of degree $4$ and genus $1$. We claim that $h^1(\Ii
_Y(4)) =1$ and $h^0(\Ii _Y(3)) =4$. Since $\chi (\Oo_Y(3))=32$, it is sufficient to prove that $h^0(\Ii _Y(4)) =4$.
We only prove the inequality $h^0(\Ii _Y(4)) \ge 4$, because the opposite inequality follows as in the proof of Example
\ref{3ma1.1}. Assume $h^0(\Ii _Y(4))\le 3$, i.e. assume that $|\Ii _Y(4)|$ is a projective space of dimension $\le 2$. This is
impossible because a subset $\Bb$ of it (isomorphic to $\PP^1\times \PP^1$) is formed by the quartic surfaces $Q_1\cup Q_2$
with
$Q_i\in |\Ii_{Y_i}(2)|$ and $\bb$ isomorphic to $\PP^1\times \PP^1$. 
\end{example}

\begin{example}\label{3ma2.2}
Let $Y=Y_1\cup Y_\cup Y_3$ be a general element of $Z(3;3;4,1;2,0;2,0)$. 

\quad {\bf Claim:} We have $h^0(\Ii_Y(4)) =2$, $h^1(\Ii _Y(4)) =1$ and all elements of $|\Ii _Y(4)|$ are union of a quadric
containing $Y_1$, the plane containing $Y_2$ and the plane containing $Y_3$.

\quad {\bf Proof of the Claim:} Let $H$ (resp. $M$) be the plane spanned by $Y_2$ (resp. $Y_3$). Since $\chi (\Oo _Y(4)) =34$,
we have
$h^1(\Ii _Y(4)) =h^0(\Ii_Y(4))-1$. Since $h^0(\Ii _{Y_1}(2))=2$ and $h^0(\Ii _{Y_i}(1)) =1$, $i=0,1$, we have
$h^0(\Ii_Y(4))\ge 2$. Fix $W\in |\Ii _Y(4)|$. To conclude the proof of the claim it is sufficient to prove that $H\cup M\subset
W$.
$Y\cap H$ is the union of $Y_2$, $2$ general points $Y_3\cap H$ and $4$ general points (Theorem \ref{l5}). Thus $h^0(H,\Ii
_{Y\cup H}(4))=0$, i.e. $H\subset W$. In the same way we get $M\subset W$.
\end{example}

\begin{lemma}\label{3ma2}
Fix integers $s\ge 1$, $d_i>0$ and $g_i\ge 0$, $1\le i\le s$, such that $d_i\ge g_i+3$ if $g_i>0$.  For all $i<j$ assume
$g_i\ge g_j$ and $d_i\ge d_j$ if $g_i=g_j$. Take a general $Y\in
Z(3;s;d_1,g_1;\dots ;d_s,g_s)$. Set $e = 35-\sum _{i=1}^{s} (4d_i+1-g_i)$ and assume $e\ge 0$. We have $h^0(\Ii_Y(4)) =\max
\{e,0\}$ and $h^1(\Ii _Y(4)) = \max \{0,-e\}$ if and only if $(s;d_1,g_1;\dots ;d_s,g_s)$ is not in one of the following
exceptional cases:
\begin{enumerate}
\item $g_i=0$ for all $i$ and either $s=2$, $d_1=4$ and $d_2=2$ or $s=3$ and $d_1=d_2=d_3=2$;
\item $s=2$ and $(d_1,g_1,d_2,g_2)\in \{(5,2,4,1),(4,1,4,1)\}$.
\end{enumerate}
\end{lemma}

\begin{proof}
We have $\binom{7}{3} =35$. By \cite{b,be4} we may assume $g_1>0$ and $s\ge 2$. Thus $d_1\ge g_1+3$. In some of the steps we
may add some restrictions on $H$ or $Q$. 

\quad (a) Assume $s=2$, $g_1>0$, $d_1\ge 3+g_1$, $g_2\ge 0$, $d_2\ge g_2+3$, $(d_1,g_1,d_2,g_2)\ne (5,2,4,1)$ and
$(d_1,g_1,d_2,g_2)\ne (4,1,4,1)$.
 Since
$d_i\ge g_i+3$ for all $i$, we get $3g_1+3g_2 \le 9$ and hence $g_1+g_2 \le 3$. Thus $d_1+d_2\le 9$ with equality allowed only
if
$g_1+g_2=3$ and hence $d_1=g_1+3$ and $d_2=g_1+3$.

\quad (a1) Assume $g_1+g_2=3$  and hence $d_i=g_i+3$ for all $i$ with $(g_1,g_2)\ne (2,1)$. Thus $g_1=3$ and $g_2=0$. Take as
$Y_2$ a smooth $Y_2\in|\Oo_Q(2,1)|$. Let $Y_1\subset \PP^3$ be a general smooth curve of degree $6$ and genus $3$. We have
$h^i(\Ii_{Y_2}(2))=0$, $i=0,1$. Set $Y:= Y_1\cup Y_2$. The residual exact sequence of $Q$ show that to prove $h^i(\Ii_Y(4))
=0$, $i=0,1$, it is sufficient to prove $h^i(Q,\Ii _{(Y_1\cap Q)\cup Y_2}(4,4)) =0$, i.e. $h^i(Q,\Ii _{Y_1\cap Q,Q}(2,3)) =0$.
This is true, because $h^1(N_{Y_1}(-2))=0$ (Remark \ref{lv2}).

\quad (a2) Assume $g_1=2$ and $g_2=0$. We have $4d_1+4d_2 \le 35$, i.e. $d_1+d_2\le 8$. Since
$d_i\ge g_i+3$ for all $i$, we get $d_1=5$ and $d_2=3$. Note that $e=3$.
Fix a plane $H$. Fix a general $Y_1\cup L\subset \PP^3$, where $Y_1\subset \PP^3$ is a smooth curve of degree
$5$ and genus $2$ and $L$ is a line. By Lemma \ref{4ma1} we have $h^1(\Ii_{Y_1\cup L}(3)) =0$.
Since $h^1(N_{Y_1}(-1)) =0$, we have $h^1(H,\Ii _{Y_1\cap H}(2)) =0$. Take a general smooth conic $D\subset H$
containing $L\cap D$. By Remark \ref{lv2} and the semicontinuity theorem it is sufficient to prove that $h^1(\Ii _{Y_1\cup
D\cup L}(4))=0$. Use the residual exact sequence of $H$.

\quad  (a3) Assume $g_1=1$ and $g_2=0$. We have $4d_1+4d_2\le 34$ and hence $7 \le d_1+d_2\le 8$. Thus $4\le d_1\le 5$.
We have $e = 2+(8-d_1-d_2)$.
Assume $(d_1,d_2)=(5,3)$. Fix a smooth $Y_2\in|\Oo_Q(2,1)|$ and then take a general $Y_1$ with $Y=Y_1\cup
Y_2\in Z(3;2;5,1;3,0)$.  We have $h^i(\Ii _{Y_1}(2))=0$, $i=0,1$ (\cite{be4}).
Using the residual exact sequence  of $Q$ we see
that it is sufficient to prove $h^1(Q,\Ii _{Y\cap Q}(4,4))=0$, i.e. $h^1(Q,\Ii _{Y_1\cap Q}(2,3))=0$. 
Since $h^0(\Ii_{Y_1}(2))=1$, $Y_1\cap Q$ is the intersection of $D$
with another
$D'\in |\Oo_Q(2,2)|$. Since $h^1(Q,\Oo_Q(0,1))=0$, the restriction map
$H^0(\Oo_Q(2,3))
\to H^0(\Oo_D(2,3))$ is surjective. Thus it is sufficient to prove $h^1(D,\Ii_{Y_1\cap Q}(2,3)) =0$. This is true, because
$D$ is an elliptic curve, $\deg (\Oo_D(2,3))=10$, $\#(Y_1\cap Q)=10$, and
$Y_1\cap Q$ may be taken as
$10$ general points of
$Q$ (Remark \ref{lv2}) and $h^0(\Ii_A(2,2))=2$ for a general $A\subset D$ with cardinality $10$.

Assume $(d_1,d_2)=(4,4)$. Take a general $Y=Y_1\cup Y_2\in Z(3;2;4,1;4,0)$. We have $h^1(\Ii_{Y_2}(2) =0$.
Take a smooth quadric $Q\supset Y_1$ with $Y_1\in |\Oo_Q(2,2)|$. Using the residual exact sequence of $Q$ we see
that it is sufficient to prove that $h^1(Q,\Ii _{Y\cap Q}(4,4))=0$, i.e. $h^1(Q,\Ii _{Y_2\cap Q}(2,2))=0$. This is true,
because we may take as $Y_2\cap Q$ $8$ general points of $Q$ (Remark \ref{lv2}).
 
The case $(d_1,d_2) =(4,3)$ is done as the previous
one.

\quad (b) Assume $s\ge 2$, $g_1>0$, $g_i=0$ and $1\le d_i\le 2$ for all $i>1$. We only do
all cases with $0\le e\le 4$, because the other ones follows from the case done either taking some of its connected components
or taking a line instead of a conic.

\quad (b1) Assume $g_1=3$ and $d_1=6$. Thus $h^i(\Ii _{Y_1}(2)) =0$, $i=0,1$, and $4d_1+1-g_1=22$. We have either $s=2$,
$d_2=2$ with $e=4$ or $s=3$ and $d_2=d_3=1$ with $e=3$. Let $Q$ be a general quadric. The set $Y_1\cap Q$ is a general union of
$12$ points of
$Q$ (Remark \ref{lv2}). If $s=2$ (resp. $s=3$) we add a general $Y_2\in |\Oo_Q(1,1)|$ (resp. general $Y_2,Y_3\in
|\Oo_Q(1,0)|$). Then we apply the residual exact sequence of $Q$.

\quad (b2) Assume $g_1=5$. Since $s\ge 2$, we have $s=2$, $d_1=8$, $d_2 =1$ and hence $e=2$. A general $Y_1$ satisfies
$h^i(\Ii _{Y_1}(3))=0$, $i=0,1$ (\cite{be4}) and $Y_1\cap H$ is a general union of $8$ points (Remark \ref{l5+}). We add a
general line
$Y_2\subset H$ and use the residual exact sequence of $H$.

\quad (b3) Assume $d_1\ge 8$ and $(d_1,g_1)\ne (8,5)$. Since $s\ge 2$ and $d_1\ge g_1+3$, we have $s=2$, $d_2=1$, $d_1=8$
and $3\le g_1\le 4$. Thus $e = g_1-3$. Take a general smooth curve $E\subset \PP^3$ with degree $6$ and genus $3$. We have
$h^i(\Ii _E(2))=0$, $i=0,1$. First assume $g_1=3$. We take a general quadric $Q$. Thus $E\cap Q$ is formed by $12$ general
points (Remark \ref{lv2}). Inside $Q$ we add  $3$ elements of $|\Oo_Q(1,0)|$, exactly two of them containing a point of
$E\cap Q$. Note hat $h^0(\Oo _Q(1,4))=10$. We apply  the residual exact sequence of $Q$ and then use the smoothing of the
union of $E$ and the two lines intersecting it. Now assume $g_1=4$.  Fix a general smooth curve $F$ of degree $5$ and genus
$2$ and a general line
$Y_2$. We have $h^1(\Ii _{F\cup Y_2}(3))=0$ (Lemma \ref{4ma1}) and hence $h^0(\Ii_{F\cup L}(3))=1$. Take a general plane $H$.
We take a smooth conic $C\subset H$ containing $3$ points of $E\cap H$ and no other point of $F\cup L$. For a general $L$ the
$3$ points $(F\cup L)\cap (H\setminus C)$ are not collinear and hence $h^i(H,\Ii _{C\cup F\cup L}(4)) =h^i(H,\Ii_{(F\cup
L)\cap (H\setminus C}(1))=0$, $i=0,1$. Apply the residual exact sequence of $H$ and then smooth $F\cup C$.

\quad (b4) Assume $d_1=7$ and $g_1=4$. Thus either $s=3$, $d_2=d_3=1$ and $e=0$ or $s=2$, $d_2=2$ and $e=1$ (assuming as always
$e\le 4$). First assume $s=3$. Let $A\subset \PP^3$ be a general curve of degree $6$ and genus $3$. Thus $h^i(\Ii _A(2)) =0$,
$i=0,1$. Fix a general secant line $L$ of $A$ and take a general quadric surface $Q\supset L$. Call $|\Oo_Q(1,0)|$ the ruling
of $Q$ containing $L$. Take general $R, R'\in |\Oo _Q(1,0)|$. Since
$A\cup L\in Z(3;1;7,4)$, it is sufficient to prove that $h^i(\Ii _{A\cup L\cup R\cup R'}(4))=0$, $i=0,1$. The residual exact
sequence of $Q$ shows that it is sufficient to prove $h^i(Q,\Ii _{Q\cap (A\cup L\cup R\cup R')}(4))=0$, $i=0,1$, i.e.
$h^i(Q,\Ii _{A\cap (Q\setminus L)}(1,4)) =0$. This is not an immediate consequence of Remark \ref{lv2}, because $Q\supset L$
and hence $Q$ is not general. We degenerate $A$ to the following nodal curve $A'$ such that $L$ is a line intersecting
transversally $A'$ at exactly $2$ points with $\dim A'\cap Q=0$. Thus it would be sufficient to prove
$h^1(Q,\Ii _{A'\cap (Q\setminus L}(1,4))=0$. Set $\{o,o'\}:= A\cap L$. Take a general line $M$ (resp. $M'$) through $o$ 
(resp. $o'$). Thus $M\cap M' =\emptyset$ and $L\cap (R\cup R') =\{o,o'\}$ (scheme-theoretically). We call $A'$ the union of
$M\cup M'$ and $4$ sufficiently general general lines $L_1$, $L_2$, $L_3$ and $L_4$ intersecting both $M$ and $M'$ (we need
$L_i\cap (R\cup R')=\emptyset$ for all $i$).
Remark \ref{smoo1} gives $A'\in Z'(3;1;6,3)$. Fix a general line $M''\subset \PP^3$. There is a unique $Q'\in |\Oo_{\PP^3}(2)|$
containing
$M\cup M'\cup M''$,
$Q'$ is a smooth and $E:= Q\cap Q'$ is a smooth element of $|\Oo_Q(2,2)|$. Hence $E$ is an elliptic curve. We take as $L-1$,
$L_2$,
$L_3$ and $L_4$ general lines intersecting each line $M$, $M'$ and $M''$. We get $8$ points of $E\setminus E\cap (R\cup E')$.
As in step (b8) below we see that $\Oo _E(1,0)\ncong \Oo _E(0,1)$ and that these $8$ points impose independent condition to
$\Oo_E(1,3)|$
and to $|\Oo _Q(1,3)|$. 

Now assume $s=2$. Fix a general secant line $L$ of $A$ and take a general plane $M$ containing $L$, a general plane $H$
and a general conic $D\subset H$. Since $D\cap M$ are two general points of $H\cap M$, we have $L\cap D=\emptyset$.
Thus $A\cup L\cup D\in Z'(3;2;7,4;2,0)$. Thus it is sufficient to prove that $h^0(\Ii _{A\cup L\cup D}(4)) =1$. Fix a general
$p\in M$. It is sufficient to prove that $h^0(\Ii _{A\cup L\cup D\cup \{p\}}(4)) =0$. Since $h^0(\Ii_A(2))=0$, it is
sufficient to prove that every $G\in |\Ii _{A\cup L\cup D\cup \{p\}}(4)|$ contains $H\cup M$. Since $G$ contains $D$ and
$A\cap H$ and $A\cap H$ are $6$ general points of $H$, $G$ contains $H$. Thus $G$ contains the line $H\cap M\ne L$. Since
$G$ contains $L$, $H\cap M$, $p$ and the $4$ general points $A\cap (M\setminus L)$, it contains $M$.

\quad (b5) Assume $d_1=7$ and $1\le g_1\le 3$. Thus $s=2$ and either $g_1=3$, $d_2=2$ and $e=0$ or $d_2 = 1$ and $e = g_1+1$.

First  assume $d_2=1$ and $g_1\ge 2$. Fix a line  $Y_2\subset H$. Let $Y_1\subset \PP^3$ be a general curve of
degree
$7$ and genus
$g_1$. Since
$g_1\ge 2$,
\cite{be4} gives $h^1(\Ii _{Y_1}(3))=0$. For a general $Y_1$ we have $Y_1\cap Y_2$. By Theorem \ref{aly1} $Y_1\cap H$ is a
general subset of $H$ with cardinality $7$. Thus $h^1(H,\Ii _{Y\cap H}(4)) = h^1(H,\Ii _{Y_1\cap H}(3)) =0$.
Now assume $g_1=1$. Let $A\subset \PP^3$ be a general smooth curve of degree $5$ and
genus
$1$. Thus
$h^i(\Ii _A(2))=0$, $i=0,1$ (\cite{be4}) and $A\cap Q$ is a general subset of $Q$ with cardinality $10$. Take $3$ distinct
elements $R, R, R'\in |\Oo _Q(1,0)|$ such that $R$ and $R'$ contain a point of $A\cap Q$. Since $A\cup R\cup R'\in
Z'(3;1;7,1)$, by semicontinuity it is sufficient to prove $h^1(\Ii_{A\cup R\cup R}(4)) =0$. Since $h^1(\Ii _A(2))=0$, the
residual exact sequence of $Q$ shows that it sufficient to prove $h^1(Q,\Ii _{(A\cup R\cup R'\cup R'')\cap Q}(4,4))=0$,
i.e. $h^1(Q,\Ii _{A\cap (Q\setminus (R\cup R'))}(1,4))=0$, which is true, because $A\cap (Q\setminus (R\cup R'))$ is a general
subset of $Q$ with cardinality $8$.

Now assume $d_2=2$ and hence $g_3=3$. Fix a general $A\in Z(3;1;6,3)$. Thus $h^i(\Ii _A(2))=0$, $i=0,1$. Fix a general line $L$
containing a point of $A$. 
Repeat the proof of  the case $s=2$ of step (e4) to get $h^0(\Ii _{A\cup L\cup D}(4))=0$ and hence $h^1(\Ii _{A\cup L\cup
D}(4))=0$.

\quad (b6) Assume $d_1=6$ and hence $1\le g_1\le 3$. Since $d_i\le 2$ for all $i$ and we assumed $e\le 4$ we have $s\ge 3$.
Since
$4d_1+1-g_1+3\cdot 5 >35$, we have $s=3$ and $d_2=d_3=1$. Thus $e=1+g_1$. Fix the plane $H$ and then take a general $Y_1\in
Z(3;1;6,g_1)$. By Theorem \ref{l5} $Y_1\cap H$ is a general subset of $H$ with cardinality
Let $o$ be a general element of $H$. We have $h^1(\Ii _{Y_1}(3))=0$, $h^0(\Ii _{Y_1}(3)) >0$ and $h^0(\Ii _{Y_1}(2))=0$
(\cite{be4}). Since $o$ is general in $H$, we get $h^1(\Ii _{Y_1\cup \{o\}}(3)) =0$. Let $L, L'\subset H$ two lines of $H$
through $o$ and disjoint from $Y_1\cap H$. Let $v\subset \PP^3$ be a general arrow of $\PP^3$ with $o$ as its reduction.
Since $W:= R\cup R'\cup v$ is a flat limit of a family of skew lines (\cite{hh0}), it is sufficient to prove
that $h^1(\Ii _{A\cup W}(4)) =0$. Since $\Res _H(A\cup W) = A\cup \{o\}$, the residual exact sequence of $H$ shows that it is
sufficient to observe that $h^1(H,\Ii _{(A\cup W)\cap H}(4)) = h^1(H,\Ii _{A\cap H}(2))=0$.

\quad (b7) Assume $d_1=5$ and hence $1\le g_1\le 2$. Either $s=3$ and $(d_2,d_3) =(2,1)$ with $e = g_1$ or $s=4$,
$d_2=d_3=d_4=1$ and $e=g_1-1$. 

\quad (b7.1) Assume $s=3$. Take a general $Y=Y_1\cup Y_2\cup Y_3$ and call $H$ the plane containing $Y_2$. By Lemma \ref{4ma1}
we have $h^1(\Ii _{Y_1\cup Y_3}(3)) =0$ and hence $h^0(\Ii _{Y_1\cup Y_3}(3)) =g_1$. Thus it is sufficient to prove that any
$G\in |\Ii_Y(4)|$ has $H$ as a component. Thus it is sufficient to observe that $h^0(\Ii _{(Y_1\cup Y_3)\cap H}(2))=0$,
because no $4$ points of $Y_1\cap H$ are collinear and $Y_3\cap H$ is a general point of $H$.

\quad (b7.2) Assume $s=4$. Fix a general $Y=Y_1\cup Y_2\cup Y_3\cup Y_4$ and call $Q$ the only quadric containing $Y_2\cup Y_3
\cup Y_4$. Since $\chi (\Oo _Y(4)) = 36-g_1$ and $h^0(\Ii _{Y_1}(2)) =g_1-1$ (\cite{be4}), it is sufficient to prove that any
$G\in |\Ii_Y(4)|$ has $Q$ as a component. $G\cap Q$ contains $Y_2\cup Y_3 \cup Y_4\cup (Y_1\cap Q)$. Thus it is sufficient to
prove
$h^0(Q,\Ii _{Y_1\cap Q}(4,1)) =0$. If $g_1=1$ it is sufficient to use Remark \ref{lv2}. Now assume $g_1=2$. Let $Q'$ be the
only quadric containing $Y_1$. For a general $Y_1$ $Q'$ is smooth, say with $Y_1\in |\Oo_{Q'}(3,2)|$. $Q'$ is a component of
$G$, because $A:= Q'\cap (Y_2\cup Y_3\cup Y_4)$ are $6$ general points of $Q'$ and hence $h^0(Q',\Ii _A(1,2))=0$.

\quad (b8) Assume $d_1=4$ and hence $g_1=1$. Either $s=3$, $d_2=d_3=2$ (but we excluded this case) or $s=4$,
$d_1=2$, $d_2=d_3 =1$ and
$e=0$. Take $s=4$. Let $Y_1\subset \PP^3$ be a smooth curve of degree $4$ and genus $1$. Fix two general points $o, o'\in
\PP^3\setminus Y_1$. We have $h^i(\Ii _{Y_1\cup \{o,o'\}}(2)) =0$, $i=0,1$. Let
$v, v'$ general arrows of $\PP^3$ with $o$ and $o'$ as their reduction. Let $Q\subset \PP^3$ be a general quadric containing
$\{o,o'\}$. $Q$ is smooth and the line spanned by $\{o,o'\}$ is not contained in $Q$. Take a general $D\in |\Oo_Q(1,1)|$
containing $\{o,o'\}$. $Y_2$ is a smooth conic. Call $R$ (resp. $R'$) the element of $|\Oo_Q(1,0)|$ containing $o$ (resp.
$o'$). Since $\Oo _Q(1,0)\cdot \Oo_Q(1,0)=0$ and $\Oo _Q(1,0)\cdot \Oo _Q(1,1) =1$ (intersection numbers), we have $R\cap
R'=\emptyset$ and $D\cup R\cup R'$ is nodal with $\{o,o'\}$ as its singular locus. For general $o$, $o'$ and $D$ we have
$Y_1\cap (D\cup R\cup R')=\emptyset$. Set
$W:= Y_1\cup D\cup R\cup R'\cup v\cup v'$. Since
$D\cup R\cup R'\cup v\cup v'\in Z'(3;3;2,0;1,0;1,0)$, by semicontinuity it is sufficient to prove that $h^i(\Ii _W(4))=0$,
$i=0,1$. Since $\Res _Q(W) =Y_1\cup \{o,o'\}$, the residual exact sequence of $Q$ shows that it is sufficient to prove that
$h^i(Q,\Ii _{W\cap Q}(4,4)) =0$, i.e. $h^i(Q,\Ii _{Y_1\cap Q}(1,3)) =0$, $i=0,1$. By Remark \ref{lv2} we may assume that
$Y_1\cap Q$ is an element of $|\Oo_E(2,2)|$, where $E$ is a general element of $|\Oo_Q(2,2)|$. $E$ is an elliptic curve. Since
$\deg (\Oo_E(1,3)) =10 = \#(Y_1\cap Q)$, it is sufficient to prove that $\Oo_D(2,2)\ncong \Oo_D(1,3)$, i.e. $\Oo _E(1,0)\ncong
\Oo_E(0,1)$. This is true  for the following reason. $Q\cong \PP^1\times
\PP^1$ and any embedding $j: E\to \PP^1\times \PP^1$ is obtained fixing two non-isomorphic degree $2$ line bundles $L_1$
and
$L_2$ and using $|L_1|$ to get the first component of $j$ and $|L_2|$ to get the second component of $j$. Since $j^\ast
(\Oo _Q(1,0))\cong L_1$ and $j^\ast (\Oo _Q(0,1)) \cong L_2$, we have $\Oo _{j(E)}(1,0)\ncong \Oo _{j(E)}(0,1)$. 
\end{proof}

Examples \ref{3ma1.1}, \ref{3ma2.0}, \ref{3ma2.1} \ref{3ma2.2} show that the omitted case in Lemmas \ref{3ma1} 
and \ref{3ma2} are exceptional cases.

\section{$n=4$}
This is the hardest part, because it is the starting case for the inductive proof on the dimension of the projective space. We
take a hyperplane $H\subset \PP^4$ and we need to the results on space curves listed or proved in section \ref{S3}.

\begin{remark}\label{lv1}
Fix integers $d$, $g$ in the Brill-Noether range for $\PP^4$, i.e. take $(d,g)\in \NN^2$ such that $5d \ge 4g+20$ and let
$Z(d,g)$ denote the irreducible component of $\mathrm{Hilb}(\PP^4)$ containing the curves with general moduli. Let
$H\subset \PP^4$ be a hyperplane. By
\cite[Corollary 2]{lv} for a general $S\subset H$ with $\# S = d$ there is a smooth $X\in Z(d,g)$ such that $X\cap H =
S$. In the range $d\ge g+4$ it is  sufficient to quote \cite{aly}. If $g=0$, this is obviously the same even when $d\le 3$.
\end{remark}

\begin{notation}\label{4ma1}
Fix $(a,q,g,d)\in \NN^4$ with $(4;d,g)$ admissible (i.e. either $g=0$ and $d>0$ or $g>0$ and $d\ge \max\{2g-1,g+4\}$)
and either $(a,q)=(0,0)$ or $(a,q)$ admissible (i.e. either $q=0$ and $a>0$ or $g>0$ and $a\ge \max \{2q-1,q+4\}$.
 We say that $(0,0)\prec (d,g)$ if and only if $g=0$.
Now assume $(a,q)\ne (0,0)$.  We say that $(a,q)\prec (d,g)$ if and only
if
$g\ge q$, $a\ge g-q+1$ and 
$d-a
\ge 2(g-q)$, except in the case $d-a=2$ in which we require that either $a\ge 4$ and $0\le g-q\le 2$ or $a\ge 3$ and $0\le
g-q\le 1$. We say that
$(a,q)\leq (d,g)$ if and only if either
$(a,q)\prec (d,g)$ or
$(a,q) =(0,0)$.
\end{notation}
Note that always $(0,0)\leq (d,g)$, while $(0,0)\prec (d,g)$ if and only if $g=0$.

\begin{notation}\label{4ma1+}
Let $\epsilon =(4;s;d_1,g_1;\dots ;d_s,g_s)$ be an admissible numerical set and let $\eta =(4;s;a_1,q_1;\dots;a_s,q_s)$ be a
generalized numerical set. We say that $\eta \prec \epsilon$ if and only if $(a_i,q_i)\prec (d_i,g_i)$ for all $i$.
We say that $\eta \leq \epsilon$ if and only if either $\eta \prec \epsilon$ or there is $i\in \{1,\dots ,s\}$ such that
$(a_j,q_j) =(d_j,g_j)$ for all $j\ne i$ and $(a_i,q_i)=(0,0)$. 
\end{notation}

The relation $\prec$ is not a partial ordering. For instance, $(1,0)\prec (4,0)\prec (5,1)$, but $(1,0)\nprec (5,1)$.
Thus the relation $\leq$ is not a partial ordering. Let $\Aa$ be a finite set of admissible generalized numerical sets. We say
that $\epsilon \in \Aa$ is \emph{maximal} among the elements of $\Aa$ for $\prec$ (resp. $\leq$) if there is no $\eta \in
\Aa\setminus \{\epsilon\}$ with $\epsilon \prec \eta$ (resp. $\epsilon \leq \eta$).

In the next observation we explain the geometrical interpretation of $\prec$ and $\leq$. Let $Y\subset H$ be a general smooth
curve of degree $a$ and genus $q$
\begin{remark}\label{4ma1.1}
Fix $(a,q,d,g)\in \NN^4$ such that $(a,q)\prec (d,g)$. Let $Y\subset \PP^4$ be a general smooth curve of degree $a$ and genus
$q$ with the convention $Y=\emptyset$ if $a=0$ (and hence $q=0$). If $(a,q)=0$ we have $g=0$ and we call $T\subset H$ a
general smooth rational curve of degree $d$. Now assume $(a,q)\ne (0,0)$. By Theorem
\ref{lv4} and Remark
\ref{lv4.1} the set
$Y\cap H$ is a general union of $a$ points. Assume for the moment $(d-a,g-q)\ne (2,3)$. In these cases there is a smooth
rational curve $T\subset H$ containing exactly $\min \{2(d-a),1+g-q\}$ points (Remarks \ref{l5+} and \ref{lv4.1}). Thus $Y\cup
T$ is a nodal curve of degree
$d$ and arithmetic genus $g$. By Remark \ref{smoo1} $Y\cup T$ is smoothable. Now assume $d-a=2$ and $g-q =3$. In this case we
take as
$T$ either the union of $2$ disjoint lines, each of them containing $2$ points of $Y\cap H$ (for this we need $\#Y\cap H\ge
4$) or a smooth conic containg $3$ points of $Y\cap H$. Thus
$Y\cup T$ is a smoothable nodal curve of degree
$d$ and arithmetic genus $g$ (Remark \ref{smoo1}).

In all cases by \cite{be4} the curve $T$ has maximal rank in $H$ and hence we may control the Hilbert function of $(Y\cap
H)\cup T$, because $(Y\cap H)\setminus (Y\cap T)$ may be considered a general subset of $H$ with cardinality $a -\#(Y\cap T)$
\end{remark}

Now we explain the geometric interpretation of $\prec$ and $\leq$ if $s>1$.

\begin{remark}\label{4ma2.1}
Fix a general $Y =Y_1\cup \cdots \cup Y_s\in Z(4;s;d_1,g_1;\cdots ;d_s,g_s)$ with $Y$
smooth and connected of degree
$d_i$ and genus $g_i$. 

First assume $\eta \prec \epsilon$. Call $T_i$ the curve called $T$ in Remark \ref{4ma1.1} for the pair $(a_i,q_i),(d_i,g_i))$
and set $T:= T_1\cup \cdots \cup T_s$. Thus
$Y_i\cup T_i$ is a nodal element of $Z'(4;1;d_i,g_i)$. We may find $T_i$ such that $T_i\cap Y_j =T_i\cap T_j$ for all $i\ne
j$. Thus $Y\cup T\in Z'(4;s;d_1,g_1;\dots ;d_s,g_s)$. We may control the postulation of $T\subset H$ using \cite{be4} and the
postulation of $Y\cap H$ using Theorem \ref{lv4} and Remark \ref{lv4.1}. We will control the postulation of $Y$ by an inductive
argument. Thus using the residual exact sequence of $H$ we will be able to control the postulation of
$Y\cup T$.

Now assume $\eta \leq \epsilon$, but $\eta \nprec \epsilon$. Fix $i\in \{1,\dots ,s\}$ such that $(a_j,q_j)=(d_j,q_j)$ for all
$j\ne i$. For $j\ne i$ set $T_j=\emptyset$. Let $T_i\subset H$ be the curve described in Remark \ref{4m1.1} for
$(a_i,q_i,d_i,g_i)$. We may find
$T_i$ such that
$T_i\cap Y_j =\emptyset$ for all $j\ne i$. Set $T:= T_i$, i.e. set $T:= T_1\cup \cdots \cup T_s$. We have $Y\cup T\in
Z'(4;s;d_1,g_1;\dots ;d_s,g_s)$. We control the postulation of $T\subset H$ using \cite{be4}.
\end{remark}

\begin{remark}\label{4m1.2}
Fix an integer $k\ge 3$ and an admissible $(d,g)\in \NN^2\setminus \{(0,0)\}$, i.e. either $g=0$ and $d>0$ or $d\ge
\max\{2g-1,g+4\}$. We take an admissible
$(a,q)\in
\NN^2\setminus (0,0)$ such that
$(a,q)\prec (d,g)$ and the ratio $(kd+1-g)/((k-1)a+1-q)$ is high (often as high as possible).

\quad (a) Assume $g=0$ and hence
$q=0$. In this case we take $a=1$ and get $(kd+1)/k$ which is $(k+1)/k$ if $d=1$ and $>2$ in all other cases. Thus in the
applications with numerical sets $\epsilon =(4;s;d_1,g_1;\dots ;d_s,g_s)$ with $s\ge 2$ we take $(0,0)$ instead of $(1,0)$ if
$(d,g)=(1,0)$, except if
$(d_i,g_i)=(1,0)$ for all $i=1,\dots ,s$ and in that case we with take $(a_1,q_1) =(1,0)$ and $(a_i,q_i) =0$ for all $i$.
Set $(4;s;a_1,q_1;\dots ;a_s,q_s)$. Note that $w_k(\epsilon) =s(k+1)$ and $w_{k-1}(\eta) =k <s(k+1)/2$.

\quad (b) Assume $1\le g\le 3$. We take $(a,q) =(g+1,0)$.. In this case we have
$(kd+1-g)/((k-1)g+k)
\ge ((k-1)g+4k+1)/((k-1)g+k) $, which is $5k/(2k-1)$ (resp. $(6k-1)/(3k-2)$, resp. $(7k-2)/(4k-3)$) for
$g=1$ (resp. $g=2$, resp. $g=3$). Thus in all cases $(kd+1-g)/((k-1)a+1-q) >2$.

\quad ({c}) Assume $g=4$ and hence $d\ge 8$. We take $(a,q)=(5,1)$. We have $(kd-3)/(5k-5) \ge (8k-3)/(5k-5)\ge 8/5$; if $k=6$
we have
$(kd-3)/(5k-5) \ge (8k-3)/(5k-5) = 9/5 > 5/3$; if $k=5$ we have $(8k-3)/(5k-5) =37/20 >9/5$. 

\quad (d) Assume $g=5$ and hence $d\ge 9$.
We take
$(a,q) =(5,1)$ and get
$(kd+1-g)/((k-1)a+1-q)
\ge (9k-4)/(5k-5) \ge 9/5$.

\quad (e)  Assume $g=6$ and $d\ge 11$. We take $(a,q)=(6,2)$  and get $(kd+1-g)/((k-1)a+1-q) \ge
(11k-5)/(6k-7)
\ge 11/6$; if $k=5$ we have $(11k-5)/(6k-7) =50/23>2$. 

\quad (f) Assume $g=7$ and hence $d\ge 13$.  We take $(a,q)=(6,2)$ 
and get
$(kd+1-g)/((k-1)a+1-q)
\ge (13k-6)/(6k-7)
\ge 13/6$.

\quad (g)  Assume $g=8$ and hence $d\ge 15$. We take $(a,q)=(7,3)$  and get $(kd+1-g)/((k-1)a+1-q) \ge (15k-7)/(7k-9) \ge
15/7$.

\quad (h) Assume $g\ge 9$ and hence $d\ge 2g-1$. We take $q =\lfloor g/2\rfloor$ and $a=2q-1$. Note that $a\ge q+4$.
We have  $(kd+1-g)/((k-1)a+1-q) \ge (k(2g-1)+1-g)/((k-1)(g+1)/2 +1 -(g+1)/2) = (k(4g-2)+2-2g)/(k(g-1)+1-2g) \ge 2$.

\quad (i) Summary: In all cases (except $(d,g) =(1,0)$, which we have explained in case (a)) we always have
$(kd+1-g)/((k-1)a+1-q)\ge 8/5$. We will use that $\binom{k+4}{4}/\binom{k+3}{4} = (k+4)/4 < 8/5$ for all $k\ge 7$.
Assume $k=6$; by step ({c}) in all cases  $(kd+1-g)/((k-1)a+1-q) \ge (k+4)/k =5/3$. Assume $k=5$; in all cases we have
$(kd+1-g)/((k-1)a+1-q) \ge 37/20>9/5$, while $5 >5/3$, while
$(k+4)/k = 9/5$.
\end{remark}

\begin{lemma}\label{qma2}
Let $\epsilon =(4;s;d_1,g_1;\dots ;d_s,g_s)$, $s\ge 2$, be an admissible numerical set with critical value $k\ge 5$ for
$\PP^4$. Let
$\Ff$ be the set of all admissible generalized numerical sets $\eta \prec \epsilon$ and with critical value $<k$.

\quad (i) $\Ff\ne \emptyset$.

\quad (ii)  Let $\eta =(4;s;a_1,q_1;\dots ;a_s,q_s)$ be a maximal element of $\Ff$. Set $a:= \binom{k-+3}{4} -\sum '
((k-1)a_i+1-q_i)$. Then $0\le a\le 2k-4$.
\end{lemma}

\begin{proof}
Assume $\Ff=\emptyset$. By Lemma \ref{nn1} it is sufficient to find $\eta =(4;s;a_1,q_1;\dots ;a_s,q_s)\prec \epsilon$
such that $\sum ' ((k-1)a_i+1-q_i) \le \binom{k+3}{4}$. If $g_i=0$ for some $i$ we take $(a_i,q_i) =(0,0)$, except the case
$g_1=\cdots =g_s=0$, where we take $(a_1,q_1) =(1,0)$ and $(a_i,q_i)=(0,0)$ for all $i>1$ (case (a) of Remark \ref{4m1.2}). In
this case we have
$\sum ' (a_i+1-q_i)= k< \binom{k+3}{4}$. Now assume $g_i>0$. We take $(a_i,q_i)$ as given by Remark \ref{4m1.2} for the
integer $g:= g_i$. By part (e) of Remark \ref{4m1.2} we have $\sum ' (ai+1-q_i) \le \frac{8}{5}(\sum _i (kd_i+1-g_i)\le
\binom{k+4}{4}$ if $k\ge 7$. For $k=6$ (resp. $k=5$) we use  part (e) of Remark \ref{4m1.2} to handle the only case
$(d,g)=(8,4)$ (resp. the only cases $(d,g)\in \{(8,4),(9,5)\}$) in which we need a better inequality.

Thus $\Ff \ne \emptyset$. Since $\Ff$ is finite, it has at least one maximal element for $\prec$. Take $\eta$ and $a$ as in
part (ii). Assume $a\ge 2k+1$.

First assume $g_i=q_i$ for some $i$ and $0<a_i<d_i$. In this case the
generalized numerical set $\tau = (4;s;b_1,e_1;\dots ;b_s,e_s)$ with $(b_j,e_j) = (a_j,e_j)$ and $(b_i,e_i) =(a_i+1,e_i)$ has
critical value $\le k-1$ and $\tau \prec \epsilon$, a contradiction. 

Now assume the existence of $i\in \{1,2\}$ such that $(a_i,q_i)=(0,0)$. By the definition of $\prec$  for $n=4$ we have
$g_i=0$, Hence we get a new admissible numerical set $\prec \epsilon$  with critical value $2$ using $(1,0)$ instead of
$(0,0)$, contradicting the maximality of $\eta$.

Now assume the existence of an index $i$ such that
$0< q_i<g_i$. By the definition of admissibility we have $a_i\ge q_i+4$ and $d_i-a_i\ge g_i-q_i$. If $g_i=q_i+1$ we may
substitute $(a_i,q_i)$ with $(a_i+1,q_i+1)$ and get a contradiction to the assumption on $a$. 
The same works if $g_i\ge q_i+2$, unless $d_i=a_i+3$ and $g_i =q_i+3$. In this case we substitute $(a_i,q_i)$ with
$(a_i+2,q_i+2)$ and get a contradiction, because $a\ge 2k-3$.

Now assume the existence of $i\in \{1,2\}$ such that $q_i =0<g_i$. If $a_i\ge 4$ we may repeat the proof of the previous
case.
Now assume $a_i \le 3$. Since $(a_i,0)$ has $a_i$ attaching points, we have $g_i\le a_i-1$. If $g_i=1$ we may use
$(a_i+1,0)$, because $d_i\ge 5$ by the admissibility. If $g_i=2$ and hence $a_i=3$ and $d_i\ge 6$ we may use $(a_i+2,1)$
instead of $(a_i,0)$, giving a contradiction, because we assumed $a\ge 2k-3$.

Since $\eta \ne \epsilon$, there is $i\in \{1,\dots ,s\}$ such that $(a_i,q_i)\prec (d_i,g_i)$ and $(a_i,g_i)\ne (d_i,g_i)$,
while we obtained a contradiction for  all possible $(a_i,q_i)\ne (d_i,g_i)$.
\end{proof}

\begin{lemma}\label{pk3}
Theorem \ref{i1} is true for $n=4$ and all admissible $\epsilon$ with critical value $3$. Moreover, $h^0(\Ii _X(3))=0$ for a 
general
$Z(\tau)$ for any admissible
$\tau$ with critical value $>3$.
\end{lemma}

\begin{proof}
With no loss of generality we may assume $g_i\ge g_j$ for all $i<j$ and $d_i\ge d_j$ for all $i<j$ such that $g_i=g_j$. By
\cite{be6} we may assume
$g_i>0$ for some
$i$. Thus
$g_1>0$. By \cite{be2} or \cite{be4} we may assume $s\ge 2$.

We have $\binom{6}{4}=15$, $\binom{7}{4} =35$ and $\binom{6}{3}=20$.
Fix $\epsilon =(4;s;d_1,g_1;\dots ;g_s)$ with critical value $3$. Lemma \ref{k=2}
gives $h^0(\Ii _X(2))=0$. Thus to prove the lemma for $\epsilon$ it is sufficient to prove that $h^1(\Ii _X(3))=0$.
Fix $i\in \{1,\dots ,s\}$ such that $d_i\ge g_i+4$. We have $3d_i+1-g_i \ge 2g_1+13$. Since $\binom{7}{4} =35$, either $s=2$
or $g_i=0$ and $d_i\le 3$ for all $i>2$.

\quad (a) Assume $s=3$ and $g_2>0$. Thus $g_1=g_2=1$, $d_1=d_2=5$, $g_3=0$ and $d_3=1$. Take a general $Y =Y_1\cup Y_2\in
Z(4;2;5,1;2,0)$. Lemma \ref{k=2} gives $h^i(\Ii _Y(2)) =0$, $i=0,1$. Take $T = T_1\cup T_2\in Z(H;2;3,0;1,0)$ with the only
restriction that $T_1$ contains the $2$ points of $H\cap Y_2$. Note that $Y\cup T\in Z'(4;3;5,1;5,1;1)$. We obviously have
$h^1(H,\Ii _{T,H}(3)) =0$ and $h^0(H,\Ii _{T,H}(3)) = 6$. Since $Y_1\cap H$ are $5$ general points of $H$, we have
$h^1(H,\Ii_{(Y\cap H)\cup T}(3)) =0$. The residual exact sequence of $H$ gives $h^1(\Ii _{Y\cup T}(3)) =0$.

\quad (b) Assume $s=2$ and $d_2\ge 2$. We start with $Y=Y_1\cup Y_2$ and we add $T=T_1\cup T_2$ with $\deg (T_1) = d_1-5$
and $\deg (T_2) = d_2-2$, with the convention $T_i=\emptyset$ if $\deg (T_i)=0$. We may cover in this way all
$(d_1,g_1,d_2,g_2)$.

\quad ({c}) Assume $s\ge 2$ and $d_2=1$. Thus $d_i=1$ for all $i>1$. Take a general $W =W_1\cup W_2\cup W_3\in
Z(4;3;4,0;1,0;1,0)$. If $1\le g_1\le 2$ we add in $H$ $s-2$ general line and a curve $Z\subset H$ with $\deg (Z) = d_1-4$ and
$p_a(W_1\cup Z) = g_1$. This is possible if $(d_1,g_1) =(5,1)$ (resp. $(d_1,g_1) =(6,2)$) taking as $Z$ the a line containing
$2$ points of
$W_1\cap H$ (resp. $2$ disjoint lines, each of them containing $2$ points of $W_1\cap H$). Taking instead of lines higher
degree smooth rational curves we get all $(d_1,g_1)$ with $1\le g_1\le 2$. If $g_1\ge 6$ we start with a general $E\in
Z(4;1;10,6)$
with $h^i(\Ii _E(2))=0$, $i=0,1$, and add in $H$ a suitable curve with lines as connected components except at most one.
If $3\le g_1\le 5$ we start with a general $F =F_1\cup F_2\in Z(4;2;7,3;1,0)$, which satisfies $h^i(\Ii _Y(2))=0$, $i=0,1$.

\quad (d) Fix $\eta =(4;s;a_1,q_1;\dots ;a_s,q_s)$ with critical value $>3$. With no loss of generality we may assume $q_i\ge
q_j$ for all $i<j$ and $a_i\ge a_j$ if $i<j$ and $q_i=q_j$. By
\cite{be2, be6} we may assume $a_1>0$ and $s\ge 2$. Set $\eta := (4;s;a_1-1,q_1-1;\dots a_s,q_s)$. Thus $w_k(\eta) =
w_k(\eta)$.
If $h^0(\Ii _T(3))=0$ for a general $T =T_1\cup \cdots \cup T_s\in Z(\eta)$, then $h^0(\Ii_{T\cup L}(3))=0$ for a general
secant line $L$ of $T$. Since $T\cup L\in Z'(\tau)$ (Remark \ref{smoo1}), we get the lemma for $\epsilon$ by semicontinuity.
Now assume $h^0(\Ii_T(3)) =1$, but assume $h^0(\Ii_{T\cup L}(3))>0$. Thus the only element $G$ of $|\Ii_T(3)|$ contains the
secant variety $\Sigma$ of $T_1$. Since $a_1-1\ge 4+q_1-1\ge 0$, $T_1$ is non-degenerate. Since $\Sigma$ is singular along
$T_!$, we have $\deg (\Sigma)\ge 3$. Since $s\ge 2$ and $T_2$ is general, we get a contradiction. Thus by the first part
it is sufficient to test all $\tau$ with $w_3(\eta)\le 33$. Such a $\tau$ has $w_3(\tau)=35$ and we tested it.
\end{proof}

\begin{lemma}\label{pk4}
Theorem \ref{i1} is true for all $\epsilon$ with critical value $4$ and $h^0(\Ii _X(4))=0$ for a  general $Z(\tau)$ for any
$\tau$ with critical value $>3$.

\end{lemma}

\begin{proof}
With no loss of generality we may assume $g_i\ge g_j$ for all $i<j$ and $d_i\ge d_j$ for all $i<j$ such that $g_i=g_j$. By
\cite{b} we may assume
$g_i>0$ for some
$i$. Thus
$g_1>0$. By \cite{be6} we may assume $s\ge 2$.

We have $\binom{7}{4}=\binom{7}{3}=35$ and  $\binom{8}{4} =70$.
We have $\binom{6}{4}=15$, $\binom{7}{4} =35$ and $\binom{6}{3}=20$.

Fix $\epsilon =(4;s;d_1,g_1;\dots ;d_s,g_s)$ with critical value $4$. Take a general $X\in Z(\epsilon)$. Lemma \ref{pk3}
gives $h^0(\Ii _X(3))=0$. Thus to prove the lemma for $\epsilon$ it is sufficient to prove that $h^1(\Ii _X(4))=0$.

For all generalized admissible numerical sets $\eta$ and $\theta$ in $\PP^4$ we write $\eta \ll
\theta$ if $\eta \prec \theta$ with the definition of $\prec$ given in Definition \ref{dd1} for $n\ge 5$, not the one
given for $\PP^4$. With this definition the proofs Lemmas \ref{ma2} and \ref{rnn12}  work verbatim if we control the
Hilbert function of the added curve $T\subset H$. We only need the case $k=4$ and we use Lemma \ref{3ma2} in $H$. We need to
justify that we never need one the $4$ exceptional cases. Since in each exceptional case, $T$, we have $h^1(H,\Ii_{T,H}(4))\le
2$ (Remark \ref{a3} and Examples \ref{3ma2.0}, \ref{3ma2.1} and \ref{3ma2.2}), it is sufficient to use that $\#Y\cap
(H\setminus T)\ge 3$.

 For $\eta =(4;s;a_1,q_1;\dots ;a_s,q_s)$ with critical value $>4$ we may adapt the proof just given or the one given in the
proof of Lemma \ref{pk3}.
\end{proof}

\begin{proof}[Proof of Theorem \ref{i1} in $\PP^4$]
Let $\epsilon =(4;s;d_1,g_1;\dots ;d_s,g_s)$ be an admissible numerical set and let $k$ be the critical value of $\epsilon$.
By Remark \ref{k=1} and Lemma \ref{k=2} to prove the theorem for $\epsilon$ we may assume $k\ge 3$. By Lemmas \ref{pk3} and
\ref{pk4}
we may assume $k\ge 5$. Le $X$ be a general element of $Z(\epsilon)$. We need to prove that $h^1(\Ii_X(k)) =0$ and that
$h^0(\Ii _X(k-1)) =0$.

\quad (a) In this step we prove that $h^1(\Ii_X(k)) =0$. Set
$\alpha:=
\binom{k+4}{4}-w_k(\epsilon)$. By the definition of critical value we have
$\alpha\ge 0$. Let
$\Ff$ be the set of all admissible generalized sets
$\eta
\prec
\epsilon$ with critical value
$<k$. By Lemma \ref{qma2} $\Ff\ne \emptyset$ and a maximal $\eta = (4;s;a_1,q_1;\dots ;a_s,q_s)\in \Ff$ satisfies
$0 \le \beta:= \binom{k+3}{4} -w_{k-1}(\eta)\le 2k-2$. The case $\alpha \ge \beta$ is done as in step (a1) of the proof of
Theorem \ref{i1} for $n>4$ given in the last section. Now assume $\beta < \alpha$. Instead of $T$ we first take a tree $J$ with
the same degree and $e:=\max \{\beta-\alpha,s-1\}$ good set of nodes, $S$. Since $\#S \le a_1+\cdots +a_s$, we may assume that
$S$ are general in $H$. Take $S'\subseteq S$ with $\#S' = \beta-\alpha$. For each $p\in S'$ let $v_p$ be a general arrow of
$\PP^4$ with $p$ as its reduction. Set $Z:=
\cup _{p\in S'} v_p$. Since each $v_p$ is general among the arrows of $\PP^4$ with $p$ as its reduction, $Z\cap H =S'$ and
$\Res _H(Z)=S'$. By Remark \ref{ma6.0} the curve $J\cup Z$ is a flat limits of a family of $\beta -\alpha+1$ disjoint
rational curves with prescribed degrees. Note that $(J\cup Z)\cap H= J$ (scheme-theoretically) and $\Res _H(J\cup Z) =S'$
(scheme-theoretically). We need that these curves passes though enough general points of $H$ to pass through the points of
$Y\cap H$ used in step (a) to get a curve $T\subset H$ such that $Y\cup T\in Z'(4;s;d_1,g_1;\dots ;d_s,g_s)$. This is true by
Remark \ref{l5+}. The curve $Y\cup J\cup Z$ satisfies $h^1(\Ii _{Y\cup J\cup Z}(k))=0$ by the residual exact sequence of $H$,
the assumption $h^1(\Ii _{Y\cup S'}(k-1)) =0$ and \cite[Lemma 2.1]{be7}.

 To prove that $h^0(\Ii _X(k-1)) =0$ we may assume $k\ge 6$ by Lemmas \ref{pk3} and \ref{pk4} and again
mimic the proof of the case $n>4$.
\end{proof}

\section{$n\ge 5$}
Now we assume $n\ge 5$ and that Theorem \ref{i1} is true in $H= \PP^{n-1}$.

\begin{definition}\label{dd1}
Let $\epsilon =(n;s;d_1,g_1;\dots ;d_s,g_s)$ and $\eta (n;s;a_1,q_1;\dots ;a_s,q_s)$ two admissible generalized
numerical sets. We say that
$\eta\prec \epsilon$ if the following conditions are satisfied:
\begin{enumerate}
\item $q_i\le g_i$ for all $i$;
\item $0\le a_i\le d_i$ if $g_i=0$;
\item assume $g_i>0$; then either $q_i=0$ and $a_i=0$ or $q_i>0$ and $0\le a_i-\max\{2q_i-1,n+q_i\} \le
d_i-\max\{2g_i-1,n+g_i\}$ or $d_i=2g_i-1>g_i+n$, $q_i=0$ and $2\le a_i\le n$.
\end{enumerate}
\end{definition}

\begin{remark}\label{ma00}
Fix $b\ge 2$ distinct points $p_1,\dots ,p_b \in \PP^m$, $m\ge 3$, and a finite set $B\subset \PP^m\setminus \{p_1,\dots
,p_b\}$ such that $\{p_1,\dots, p_b\}\cup B$ is in linearly general position in $\PP^m$, i.e. any $k\le m+1$ points of 
of $\{p_1,\dots, p_b\}\cup B$ are linear independent. The line $L_1$ spanned by $p_1,p_2$ contains no point of $\{p_3,\dots
,p_b\}\cup B$. If $b\ge 3$ there is a line $L_2$ containing $b_3$ and a point of $L_1$ and with $L_2\cap (B\cup
\{p_1,p_2,p_4,\dots ,p_b\})=\emptyset$. If $b\ge 4$ there is a line $L_3$ containing $b_4$ and a point of $L_2$ and with
$L_2\cap (B\cup
\{p_1,p_2,p_3,p_5,\dots ,p_b\})=\emptyset$. In this way we get a connected nodal curve $L_1\cup \cdots \cup L_{b-1}$
of degree $b-1$ and arithmetic genus $0$ such that $\{p_1,\dots ,p_b\}\subset L_1\cup \cdots \cup L_{b-1}$
and $(L_1\cup \cdots \cup L_{b-1})\cap B=\emptyset$. Smoothing these curves for all integers $d\ge x-1>0$ we may obtain a
smooth rational curve
$T\subset
\PP^m$ with degree $d$ containing $x+1$ points of $\PP^m$ in linear general position.
\end{remark}

The following observation explains the reason for Definition \ref{dd1} for $n\ge 5$.

\begin{remark}\label{ma4}
Fix admissible numerical sets $(n;1;d,g)$ and $(n;1;a,q)$ such that $(n;1;a,q) \prec (n;1;d,g)$ and $(a,q)\ne (d,g)$. Fix a
general
$Y\in Z(n;1;a,q)$. In particular $Y$ is transversal to $H$. We want to find an admissible (i.e. with degree and genera
admissible for $\PP^{n-1}$)
$T\subset H$ such that
$Y\cup T \in Z'(n;1;d,g)$. 
$Y\cup T \in Z'(n;1;d,g)$. Note that this is always possible in the omitted case $a=q=0$, because in $\PP^n$ we require that
either $g=0$ or $d\ge \max \{2g-1,g+n\}$, while in $\PP^{n-1}$ we require that either $g=0$ or $d\ge \max \{2g-1,g+n-1\}$.

\quad (a) Assume $2q-1\ge q+n$, i.e. $q\ge n+1$. Thus $2g-1\ge g+n$. Write $a = 2q-1+f$ and $d = 2g-1+e$ with $e\ge
f\ge 0$. Thus $d-a=\deg (T)=e-f$ and $d-a\ge 2(g-q)$.

\quad (a1) Assume $a\ge d-a-1$. In this case by Remark \ref{ma00} we may take as $T$ a smooth
rational curve containing exactly $g-q+1$  points of $Y\cap H$ even if $Y\cap H$ does not contain $g-q+1$ general points of
$H$ (Remark \ref{ma00} only requires that the points of $Y\cap H$ are in linear general position.

\quad (a2) Assume $a\le d-a-2$. If $d-a\ge g-q+n-1$ we may take as $T$ a smooth curve of degree $d-a$ and genus $g-q$
containing exactly one point of $Y\cap H$. Assume $d-a\le g-q+n-2$. Since $d-a\ge 2(g-q)$, we have $n-2\ge g-q$. Thus $d-a\le
2n-4$ and hence $a\le 2n-4$. By Remark \ref{aly10} to find $T$ containing all points of $Y\cap H$ and with $\deg
(T)=d-a$ and
$p_a(T) =g-q-a+1$ it is sufficient to have $d-a\ge (g-q-a+1)+n-1$ and $d-a\ge 2(g-q-a+1)-1$. Obviously $d-a\ge g-q-a+1+n-1$,
while $d-a\ge 2(g-q-a+1)-1$, because $a>0$ and $d-a\ge 2(g-q)$.

\quad (b) Assume $2g-1\ge g+n$ and $2q-1\le q+n-1$, i.e. $g\ge n+1 \ge q+2$. Since $d-2g+1 \ge a-q-n$, we have $d-a \ge
2g-q-n-1$. By Condition (3) in Definition \ref{dd1} we have $q>0$ and hence $a\ge q+n\ge n+1$. We take $T$ of genus $g-q-n$
containing  $n+1$ points of $Y\cap H$ (to find $T$ containing these points we use that any two $(n+1)$-ples of points of $H$ in
linear general position are projectively equivalent). We have $d-a\ge 2g-q-n-1 \ge 2(g-q-n)-1$. We have $d-a\ge 2g-q-n-1 \ge
(g-q-n)+n$, because $g\ge n+1$.

\quad({c}) Assume $2g\le g+n$, i.e. $g\le n$. By assumption $d-a\ge g-q$.  If $a\ge g-q+1$ we may take as $T$ a smooth
rational curve of degree $d-a$ containing $g-q+1$ points of $Y\cap H$ (Remark \ref{ma00}). Assume $a\le g-q$. Since $g\le n$
and $(n;1;a,q)$ is admissible, $q=0$. This is excluded by Condition (3) in Definition \ref{dd1}.
\end{remark}

\begin{definition}\label{ma3.00}
Let $\Aa$ be the set of all quadruples $(a,q,d,g)$ such that either $a=q=0$ or $(n;1;a,q)$ is admissible in $\PP^n$,
either $(d,g)=(0,0)$ or $(n;1;d,g)$ is admissible in $\PP^n$, $(a,q) =(0,0)$ if $(d,g)=(0,0)$ and $(n;1;a,q)\prec (n;1;d,g)$
if
$(a,q)\ne (0,0)$. For each $(a,q,d,g)\in \NN^4$ we define the following non-empty subset $\tau(a,q,d,g)\subset \NN^2$.
Set $\tau (0,0,d,g):= \{(d,g)\}$. If $(a,q) \ne (0,0)$ let $\tau (a,q,d,g)\subset \NN^2$ be the set of all $(b,e)\in \NN^2$
with the following properties. If $(a,q) =(d,g)$, set $\tau (a,q,d,g) =\{d,g)\}$. In all other cases each $(b,e)\in \tau
(a,q,d,g)$ has $b=d-a$, $0\le e\le g-q$ and $(n-1;1;d-a,e)$ admissible in $\PP^{n-1}$. We have $(d-a,e)\in \tau (a,q,d,g)$ if
and only if
$0\le e\le g-q$,
$g-q+1 \le \min \{a/2,(d-a)/2\}$ and $(n-1;1;d-a,e)$ admissible in $\PP^{n-1}$. Remark \ref{ma4} shows that $\tau (a,q,d,g)
\ne \emptyset$.
\end{definition}

\begin{remark}\label{ma7}
Let $\epsilon =(n;s;d_1,g_1;\cdots ;d_s,g_s)$ be an admissible  numerical set and let $\eta =(n;s;a_1,q_1;\dots ;a_s,q_s)$ be
an admissible generalized numerical set such that $\eta \prec \epsilon$. Fix a general $Y=Y_1\cup \cdots \cup Y_s\in
Z(n;s;a_1,q_1;\dots ;a_s,q_s)$ with
$Y_i$ of degree $d_i$ and genus $g_i$. Since $Y$ is general, $Y$ is transversal to $H$ and we may apply Remark \ref{ma4} to
each
$Y_i\cap H$. We saw in Remark \ref{ma4} the existence of $T' =T'_1\cup \cdots \cup T'_s\subset H$ with each $T_i$ smooth of
degree
$d_i-a_i$ and some genus
$b_i$ such that $\#(Y_i\cap T_i) = 1+g_i-q_i-b_i$. Thus $Y_i\cup T'_i\in Z'(n;1;d_i,g_i)$. For a general
$Y$ we may move all components $Y_j\ne Y_i$ keeping fixed $Y_i$. Thus we may find $T'$ such that $Y_i\cap T'_j=\emptyset$ for
all $i\ne j$. Thus
$Y\cup T' \in Z'(n;s;d_1,g_1;\dots ;d_s,g_s)$.
\end{remark}

Consider the following Assertion $R(n,k)$, $n\ge 4$, $k\ge 2$:

\quad {\bf Assertion} $R(n,k)$, $n\ge 4$, $k\ge 2\ :$ Fix an admissible numerical set $\epsilon =(n;s;d_1,g_1;\dots
;d_s,g_s)$, $s\ge 2$, such that $\binom{n+k}{n}\le w_k(\epsilon) \le \binom{n+k}{n}+2k$ and set $e:=
w_k(\epsilon)-\binom{n+k}{k}$. Fix a general
$S\subset  H$. There are $Y_i\in Z(n;1;d_i,g_i)$ such that $W:= Y_1\cup \cdots \cup Y_s$ is nodal, $\mathrm{Sing}(W) =S$,
 each $Y_i$ is smooth and $h^i(\Ii _W(k))=0$, $i=0,1$.

\begin{lemma}\label{rnn1}
$R(n,2)$ is true.
\end{lemma}

\begin{proof}
 Fix an admissible numerical set $\epsilon =(n;s;d_1,g_1;\dots ;d_s,g_s)$
such that $\binom{n+2}{n}\le w_2(\epsilon) \le \binom{n+2}{n}+4$. Set $e:= w_2(\epsilon)-\binom{n+2}{n}$. The case $e=0$
is true by Lemma \ref{k=2}. For $0<e\le 2k=4$ we use that ant two sets of $e$ points of $H$ in linear general position are
projectively equivalent and first we apply   Lemma \ref{rnn0} and then the proof of Lemma \ref{k=2}.
\end{proof}

\begin{lemma}\label{ma2}
Let $\epsilon =(n;s;d_1,g_1;\dots ;d_s,g_s)$, $s\ge 2$, be an admissible numerical set with critical value $k\ge 3$. Let
$\Ff$ be the set of all admissible generalized numerical sets $\eta \prec \epsilon$ and with critical value $<k$.

\quad (i) $\Ff\ne \emptyset$.

\quad (ii)  Let $\eta =(n;s;a_1,q_1;\dots ;a_s,q_s)$ be a maximal element of $\Ff$. Set $a:= \binom{n+k-1}{n} -\sum '
((k-1)a_i+1-q_i)$. Then $0\le a\le 2k$.
\end{lemma}
\begin{proof}
Assume $\Ff=\emptyset$. Let $\epsilon _i$ be the generalized numerical set $(n;s;b_1,e_1;\dots ;b_s,e_s)$ with $(b_j,e_j)
=(0,0)$ for all $j\ne i$ and $(b_i,e_i)=(d_i,g_i)$. Since $(d_i,g_i)\ne (0,0)$, $\epsilon _i$ is ann admissible generalized
set and
$\epsilon _i\prec \epsilon$. Since $\Ff=\emptyset$ by assumption, Lemma \ref{nn1} gives $(k-1)d_i+1-g_i\ge \binom{n+k-1}{n}+1$.
Thus $(d_i,g_i)\notin \{(1,0),(2,0)\}$. If either $g_i=0$ or $d_i\ge g_i+n$, then $(n;1;1,0)\prec (n;1;d_i,g_i)$
and hence the generalized numerical set $\eta  _i=(n;u_1,v_1;\dots ;u_s,v_s)$ with $(u_j,v_j) =(0,0)$ for all $j\ne i$ and
$(u_i,v_i)=(1,0)$ satisfies $\eta _i\prec \epsilon$  and it has critical value $1$, contradicting the
assumption
$\Ff =\emptyset$. If $d_i\ge 2g_i-i>g_i+n$ we use $(2,0)$ instead of $(1,0)$ to get a contradiction.

Thus $\Ff \ne \emptyset$. Since $\Ff$ is finite, it has at least one maximal element for $\prec$. Take $\eta$ and $a$ as in
part (ii). Assume $a\ge 2k+1$. First assume $g_i=q_i$ for some $i$ and $0<a_i<d_i$. In this case the
generalized numerical set $\tau = (n;s;b_1,e_1;\dots ;b_s,e_s)$ with $(b_j,e_j) = (a_j,e_j)$ and $(b_i,e_i) =(a_i+1,e_i)$ has
critical value $\le k-1$ and $\tau \prec \epsilon$, a contradiction. Now assume the existence of an index $i$ such that
$0< q_i<g_i$. By the definition of $\prec$ we have $0 \le a_i-\max \{2q_i-1,n+q_i\}\le d_i-\max \{2g_i-1,n+g_i\}$.
If $a_i-\max \{2q_i-1,n+q_i\}< d_i-\max \{2g_i-1,n+g_i\}$, we obtain a contradiction substituting $(a_i,q_i)$
with $(a_i+1,q_i)$. Now assume $a_i-\max \{2q_i-1,n+q_i\}= d_i-\max \{2g_i-1,n+g_i\}$. If $a_i\le g_i-1$ and $a_i\ge n+q_i$,
we may use $(a_i+1,q_i+1)$, keeping fixed the other $(a_j,q_j)$. Now assume $0<q_i<a_i$ and  $a_i\le n+q_i-1$. Thus $a_ige
2q_i-1$ In this case instead of $(a_i,q_i)$ we use $(a_i+2,q_i+1)$. Since $a\ge 2k-1$, the new numerical set has critical
value $\le k-1$, a contradiction. Now assume that for all $i$ either $(a_i,q_i) = (d_i,q_i)$ or $(a_i,q_i)=(0,0)$. Since $\tau$
has critical value $<k$, there is $i$ such that $(a_i,q_i) =(0,0)$. First assume $g_i=0$. In this case taking $(1,0)$ instead
of $(0,0)$ we get another element of $\Ff$, a contradiction. Now assume $g_i>0$ and hence $d_i\ge n+1$. In this case
substituting
$(2,0)$ to $(0,0)$ we get another element of $\Ff$ (because we assumed $a\ge 2k+1$), a contradiction.
\end{proof}

\begin{lemma}\label{hor2}
Fix integers $n\ge 5$, $k\ge 3$ and take $\eta =(n;s;a_1,q_1;\dots ;a_s,q_s)$ as in part (ii) of Lemma \ref{ma2}.
Then $a_1+\cdots +a_s\ge 8k-1$. If $n(n+1)(k+1) \le k(n+k-1)$, then $a_1+\cdots +a_s\ge \binom{n+k-2}{n-2}$.
\end{lemma}

\begin{proof}
With no loss of generality we may assume $a_i\ne 0$ is and only if $1\le i\le e$. Since
$\eta$ is admissible and $w_k(\eta)\ge \binom{n+k}{n}$, we get $(k-1)(a_1+\cdots +a_s)\ge \binom{n+k}{n}$.
Since $n\ge 5$, induction on $k$ starting with the case $k=3$ gives $(k-1)(8k-1)\ge \binom{n+k}{n}$. For the second inequality
just use that if $a_i\ne 0$ we have $ka_i+1-q_i\le (k+1)a_i$ and that $(k+1)\binom{n+k-2}{n-2} \le \binom{n+k-1}{n}$ if and
only if $n(n+1)(k+1) \le k(n+k-1)$.
\end{proof}

\begin{lemma}\label{rnn12}
Fix in integers $n\ge 5$ and $k\ge 3$. Assume Theorem \ref{i1} in $\PP^{n-1}$, Theorem \ref{i1} in $\PP^n$ for all admissible
numerical sets with critical value $<k$ and $R(n,k-1)$. Then $R(n,k)$ is true.
\end{lemma}

\begin{proof}
Fix an admissible numerical set $\epsilon =(n;s;d_1,g_1;\dots
;d_s,g_s)$, $s\ge 2$, such that $\binom{n+k}{n}\le w_k(\epsilon) \le \binom{n+k}{n}+2k$ and set $e:=
w_k(\epsilon)-\binom{n+k}{k}$. Fix a general
$S\subset  H$. Let $\Hh$ be the set of all admissible generalized numerical  set $\eta \prec \epsilon$ such that
$\binom{n+k-1}{n} \le w_{k-1}(\eta) \le \binom{n+k-1}{n}+2k$.

Fix any $\eta =(n;s;a_1,q_1;\cdots ;a_s,q_s)\in \Hh$ and set $b:= w_{k-1}(\eta)-\binom{n+k-1}{n}$. Fix a general $S'\subset H$
with $\#S'=e$ and take $Y =Y_1\cup \cdots \cup Y_s$ satisfying $R(n,k-1)$ for $\eta$ (with $Y_i=\emptyset$ if
$(a_i,q_i)=(0,0)$). We  may  assume that $S\cup S'$ is a general subset of $H$. We
may assume $Y\cap H\supset S\cup S'$, because $\#(S\cup S') =b+e\le 4k-4$ and $a_1+\cdots +a_s \ge  8k-1$ (Lemma
\ref{hor2} and Remark \ref{l5.2}).  Let $A\subset H$ be a general union of $e$ arrows of $H$ with $A_{\red}=S'$. Note that
$Y\cup A\in Z'(\eta)$. Take an admissible generalized numerical set $\tau (n-1;s;b_1,e_1;\dots ;b_se_s)$ such that
the pair $(\tau,\eta)$ is associated. By Remark \ref{ma7} there is $T=T_1\cup\dots \cup T_s \in Z(H,\tau)$ such that $Y\cup
T\cup A\in Z'(\epsilon)$. By Theorem \ref{i1} in $H =\PP^{n-1}$, $T$ has maximal rank. Since $\sum '(kb_i+1-e_i) + \deg
(Y)-\#(Y\cap T)+e =\binom{n+k-1}{n-1}$, we get $h^1(H,\Ii _{T,H}(k-1)) =0$ and in particular $h^0(H,\Ii _{T,H}(k-1)) \ge 0$.

\quad {\bf Claim 1} Let $Z$ be the union of the $e$ connected components $(Y\cap H)\cup A$ with $S'$ as their reduction.
We claim
that $h^1(H,\Ii _{T\cup Z,H}(k)) =0$.

\quad {\bf Claim 1:} First assume $h^1(H,\Ii _{T,H}(k-1)) =0$, $2k(k+1)\le \binom{n+k-2}{n-2}$ and $(n-2,k)\notin
\{(3,4),(4,3),(4,4)\}$. Let
$M\subset H$ be a general hyperplane of $H$. We specialize $S'$ to a general subset of $M$ with cardinality $e$ and $Z$ to
a union $Z'$ of $e$ general
$2$-points of general planes of $M$. Since $e\le 2k$ and $2k(k+1)\le \binom{n+k-2}{n-2}$ for $n\ge 5$, the
Alexander-Hirschowitz theorem say that $e$ general
$2$-points of
$M$ gives independent conditions to
$H^0(\Oo_M(k))$ (\cite{AHinv,AHjag,BO}). In particular $h^1(M,\Ii_{Z',M}(k))=0$. Using \cite[Lemma 2.1]{be7} we get
$h^1(M,\Ii_{Z'\cup (T\cap M),M}(k))=0$. The residual exact sequence of $M$ in $H$ gives $h^1(H,\Ii _{T\cup Z',H}(k)) =0$.
Claim 1 follows from the semicontinuity theorem. Now we only assume $h^1(H,\Ii _{T,H}(k-1)) =0$. In all exceptional cases $Z'$
works, because it is only union of $2$-points of general planes of $H$. Now assume $h^1(H,\Ii _{T,H}(k-1)) >0$. Since $T$ has
maximal rank in $H$,
$h^0(H,\Ii _{T,H}(k-1))=0$ and $h^1(H,\Ii _{T,H}(k-1))>0$, i.e. assume that $T$ has critical value $k$ in $H$. One could exclude
this possibility for large
$k$ by the last sentence of Lemma \ref{hor2}, but only for large $k$. At this point we use how we proved that $h^1(H,\Ii _T(k))
=0$ with $T$ with critical value $k$. Let $\Aa$ denote the set of all admissible generalized numerical sets $\alpha$ (in
$\PP^{n-1}$) such that $\alpha \prec \tau$ (in $\PP^{n-1}$) and $\alpha$ has critical value $<k$. We have $\Aa \ne \emptyset$
and any maximal element $\alpha$ of $\Aa$ has $w_{k-1}(\alpha) =\binom{n+k-2}{n-1}-f$ with $0\le f\le 2k-2$. To prove that a
general $T\in Z(\tau)$ satisfies $h^1(H,\Ii _{T,H}(k)) =0$, without using $R(n-1,k)$ it is sufficient (see step (a) of Theorem
\ref{i1} for $n\ge 5$ in the next section) to prove that $w_k(\tau) \le \binom{n+k-1}{n-1}-f$. We use $M$ as hyperplane of
$H$.  By
\eqref{eqhor1}) this is true if
$\#(Y\setminus Y\cap T)+e\ge f$, which is an obvious inequality by Lemma \ref{hor2}. Using the Alexander-Hirschowitz theorem
we also have
$h^1(M,\Ii_{Z',M}(k)) =0$. Using \cite[Lemma 2.1]{be7} we get $h^1(M,\Ii _{(T\cap M)\cup Z',M}(k)) =0$. Then using the
residual exact sequence of $M$ in $H$ we get $h^1(H,\Ii _{T\cup Z'}(k))=0$.

To conclude use Claim 1 and \cite[Lemma 2.1]{be7}.
\end{proof}

\begin{lemma}\label{5may1}
Let $Z\subset \PP^m$, $m\ge 3$, be a general union of $z:= \lfloor (\binom{m+3}{3}-2)/3\rfloor$ planar $2$-points. Then
$h^1(\Ii _Z(3)) =0$.
\end{lemma}

\begin{proof}
Assume that the lemma fails and let $w<z$ the maximal integer such that $h^1(\Ii _W(3)) =0$ for a general union of $w$ planar
$2$-points. Let $\Bb$ denote the scheme-theoretic base locus of $|\Ii_W(3)|$. By Lemma \ref{rplan1} $\dim \Bb \ge m-2$. 
\end{proof}

\section{End of the proof for $n\ge 5$}

In this section we assume $n\ge 5$ and that Theorem \ref{i1} is true in $H =\PP^{n-1}$. We fix a positive integer $k$ and
prove Theorem \ref{i1} for all admissible numerical sets with critical value $k$. By Remark \ref{k=1} and Lemma \ref{k=2} Theorem \ref{i1} is true for all admissible numerical sets with critical value $1$ and
$2$. Thus it is sufficient to test all numerical sets with critical values $\ge 3$. We fix an integer $k\ge 3$. By induction
on the critical value we may assume that Theorem \ref{i1} is true for all numerical sets with critical value $<k$.

\begin{lemma}\label{rma1}
Fix integers $n\ge 5$ and $k\ge 2$. Let $\epsilon = (n;s;d_1,g_1;\dots;d_s,g_s)$ be an admissible numerical 
such that $w_k(\epsilon)\ge \binom{n+k}{n}$. Then there is an admissible  generalized  set $\eta \prec \epsilon$ such
that
$\eta \prec \epsilon$ and
$\binom{n+k}{n}\le w_k(\eta)\le \binom{n+k}{n}+2k$.
\end{lemma}

\begin{proof}
Let $\Ff$ be the set of all admissible generalized  sets $\eta$ such that $\eta \prec \epsilon$ and $w_k(\eta)\ge
\binom{n+k}{n}$. Note that $\epsilon \in \Ff$ and that $\Ff$ is finite. Let $\eta$ be a minimal element of $\Ff$. To conclude
it is sufficient to prove that $w_k(\eta)\le \binom{n+k}{n}+2k$. Set  $\eta =(n;s;a_1,q_1;\dots ;a_s,q_s)$ and $a:=
\sum ' (ka_i+1-q_i)-\binom{n+k-1}{n}$. To prove the lemma it is sufficient to prove that $a\le 2k$. Assume $a\ge 2k+1$. Since
$w_k(\eta)\ge
\binom{n+k}{n}$, there is $i\in \{1,\dots ,s\}$ such that $(a_i,q_i) \ne (0,0)$. Call $\tau =(n;s;b_1,e_1;\dots ;b_s,e_s)$ an
admissible generalized set with $(b_j,e_j)=(a_j,q_j)$ for all $j\ne i$ and the following $(b_i,e_i)$. We show what
we need about
$(b_i,e_i)$ to obtain a contradiction to the minimality of $\Ff$. We need $kb_i+1-e_i\ge ka_i+1-q_i -2k-1$ to get $w_k(\tau)\ge
\binom{n+k}{k}$.
 We need $(b_i,q_i)\prec (a_i,q_i)$
(so that $\tau \prec \eta$) and $(b_i,q_i)\ne (a_i,q_i)$, to contradicts the minimality of $\eta$.

First assume $g_i=0$ and
hence $q_i=0$. In this case we take $b_i=a_i-1$ and $e_i=0$. We have $w_k(\eta) =w_k(\tau) -\binom{n+k}{n} =a-k-1$ if $a_i=1$
and $w_k(\eta) =w_k(\tau) -\binom{n+k}{n} =a-k$ if $a_i\ge 2$; note that $\tau$ is a generalized numerical set if $d_i=1$.
Obviously $\tau\prec \eta$ and $\tau \prec \epsilon$.

Now assume $q_i>0$, so that $0
\le a_i-\max\{2q_i-1,q_i+n\}\le d_i-\max\{2q_i-1,q_i+n\}$. If $a_i> \max\{2q_i-1,q_i+n\}$ we take $(b_i,q_i) =(a_i-1,q_i)$ and
get $w_k(\tau)-w_k(\eta) =k-k$, $\tau \prec \eta$ and $\eta \prec \epsilon$. If $a_i= \max \{2q_i-1,q_i+n\}$ we take
$b_i=a_i=2$ and $e_i=q_i-1$. Thus $w_k(\eta) -w_k(\tau) =w_k(\tau) =a-2k+1$, $\tau \prec \eta$ and $\eta \prec \epsilon$.

Now assume $g_i>0$ and $q_i=0$. We take $b_i=0$. To get $\tau \prec \eta$ it is sufficient to take $0\le b_i\le a_i$. If $0<
a_i-n
\le d_i-\max\{2g_i-1,n+g_i\}$ or $d_i=2g_i-1>g_i+n$, $q_i=0$ and $2\le a_i\le n$ we take $(b_i,e_i)=(a_i-1,0)$.
If $d_i=2g_i-1>g_i+n$ and $3\le a_i\le n$, we take $b_i=a_i-1$. If $d_i=2g_i-1>g_i+n$ and $a_i=2$ we take $b_i=0$.
\end{proof}
Thus from now on we assume $k\ge 3$, that Theorem \ref{i1} is true for all admissible numerical sets with critical value $<k$
and that $R(n,t)$ is true for all $t<k$ (use Lemma \ref{rnn1}).

Fix a numerical set $\epsilon = (n;s;d_1,g_1;\cdots ;d_s,g_s)$ with critical value $k$. Since $Z(\epsilon)$ is irreducible, to prove Theorem \ref{i1} for $\epsilon$
it is sufficient to find some $X',X''\in Z(\epsilon)$ such that $h^1(\Ii _{X'}(k)) =0$ and $h^0(\Ii _{X''}(k-1)) =0$.

\subsection{$h^1(\Ii _{X'}(k))=0$}

\begin{lemma}\label{v1}
Take an admissible  generalized numerical set $\eta$ with critical value $k-1$ and $e: = \binom{n+k-1}{n-1}-w_{k-1}(\eta)\le
2k$. Take a general $Y\in Z(\eta)$. Then there is $S\subset Y\cap H$ such that $\#S= e$ and $h^i(\Ii_{Y\cup W}(k-1))=0$,
$i=0,1$, where $W = \cup _{p\in S} v_p$ and $v_p$ is a general arrow of $H$ with $p$ as its reduction.
\end{lemma}

\begin{proof}
We stress that $v_p \subset H$ for all $p\in S$. To prove the lemma we may assume $e>0$. Assume the existence of $A\subset
Y\cap H$ with
$\#A =f$ and $h^0(\Ii _{Y\cup E}(k-1))=e-f$, where $E$ is a general union of $f$ arrows of $H$ with $E_{\red} =A$, but that
for any $p\in Y\cap H\setminus A$ and any arrow $v_p\subset H$ we have $h^0(\Ii _{Y\cup E\cup v_p}(k-1)) =e-f$, i.e. $v_p$ is
in the base locus of $|\Ii_{Y\cup E}(k-1)|$, i.e. $\Bb$ contains $(2p,H)$ for all $p\in Y\cap H\setminus A$. Set
$t:=
\deg (Y)=a_1+\cdots +a_s$ with $a_i\ge a_j$ for all $i\ge j$. For a general $Y$ we may assume that each $Y_i\cap H$ is in
uniform position. Thus
$Y_i\cap H\subset \Bb$ if $Y_i\nsubseteq A$. Taking as $A$ first $a_1-1$ points of $Y_1\cap H$, then , we get that $\bb$
contains $\cup _{p\in Y\cap H}(2p,H)$, unless $a_1+\cdots +a_s \le s+2k$. In this case there is $Y'\subset Y$ with $Y\setminus
Y$ union of at most $2k$ lines, such that  $\Bb$  contains $\cup _{p\in Y'\cap H}(2p,H)$. Thus there is $G\in |\Ii_Y(k-1)|$
with this property.
Moving $H$ we see that $G$ is singular at each point of $Y'$. Since the singular locus of a hypersurface is the zero-locus of
the partial derivatives of its equation, we get $|\Ii_{Y'}(k-2)|\ne \emptyset$, contradicting Theorem \ref{i1} for the
critical value $k-1$.
\end{proof}
Now we will prove the existence of $X'\in
Z(\epsilon)$ such that $h^1(\Ii _{X'}(k)) =0$. Set $\alpha := \binom{k+4}{4} -\sum
_{i=1}^{s} [(kd_i+1-g_i]$. We have $\alpha \in \NN$. By Lemma \ref{ma2} there is an  admissible generalized numerical set 
$(n;s;a_1,q_1;\dots ;a_n,q_s) \prec (n;s;d_1,g_1;\cdots ;d_s,g_s)$ with critical value $k-1$
and $\beta := \binom{n+k-1}{n} -\sum ' [(k-1)a_i+1-q_i)] \le 2k$.  Take a general $Y =
Y_1\cup \cdots \cup Y_s\in Z(n;s;a_1,q_1;\dots ;a_s,q_s)$ with $Y_i=\emptyset$ if $a_i=0$ and $Y_i\in Z(n;1;a_i,q_i)$ if
$d_i>0$. By the inductive assumption we have $h^1(\Ii _Y(k-1)) =0$ and $h^0(\Ii _Y(k-1)) =\beta$. 

\quad (a) Assume $\beta \le \alpha$.
Take  $T =T_1\cup \cdots \cup T_s\in Z(H;s;d_1-a_1,b_1;\dots ;d_s-a_s,b_s)$, where the integers $b_i$ and the curves $T_i$ have
the following properties. If $d_i=g_i$, then $b_i =0$ and $T_i=\emptyset$. Now assume $d_i>a_i$. By the definition of $\prec$
discussed in Remark \ref{ma7} there are integers $b_i$ and $e_i$ such that $1\le e_i\le a_i$, $q_i+b_i+e_i-1 =g_i$,
the extended numerical set $(n-1;s;d_1-a_1,b_1;\dots ;d_s-a_s,b_s)$ is admissible for $\PP^{n-1}$. By Remark \ref{ma7} we may
find $T= T_1\cup \cdots \cup T_s\in  Z'(H;s;d_1-a_1,b_1;\dots ;d_s-a_s,b_s)$ with the following properties. If
$d_i=a_i$ (and hence
$q_i=g_i$ by the definition of $\prec$) set $T_i =\emptyset$. Now assume $d_i>a_i$. let $T_i$ be a general of
$Z(H;1;d_i-a_i,0)$ containing exactly $g_i-q_i$ points of $Y_i\cap H$; we may do this because $g_i-q_i\le d_i-a_i$ by the
definition of $\prec$, a general smooth rational curve of degree $t$ contains $t+1$ general points of $H$ and the set $Y\cap
H$ is general in $H$.  Set $W:= Y\cup T$ with $X_i = Y_i\cup T_i$. By Remark \ref{smoo1} we have $W\in Z'(n;s;d_1,g_1;\dots
;d_s,g_s)$. By the semicontinuity theorem for cohomology to prove the existence of $X''$ it is sufficient to prove that
$h^1(\Ii_W(k)) =0$. Since $h^1(\Ii _Y(k-1)) =0$, the residual exact sequence of $H$ shows that it is sufficient to prove
$h^1(H,\Ii _{W\cap H}(k)) =0$. The scheme
$W\cap H$ is the union of $T$ and the set $Y\cap (H\setminus T)$. By assumption $h^1(H,\Ii _{T,H}(k)) =0$. Since $\beta \le
\alpha$, we have $\# Y\cap (H\setminus T)\le h^0(H,\Ii _{T,H}(k))=0$. Now we use the assumption that $\eta$ is admissible and
in particular $a_i \ge 2q_i-1$ for all $i$. We apply several times \cite[Lemma 2.1]{be7}.

\quad (b) Assume $\beta >\alpha$. Thus $\alpha \le 2k-1$. Fix a general $S\subset H$ with $\#S=\beta$. By Lemma \ref{rma1} and
$R(n,k-1)$ there is an admissible generalized numerical set
$\tau = (n;s;b_1,e_1;\dots ;b_s,e_s)\prec \epsilon$ with $w_{k-1}(\tau) =\binom{n+k-1}{n}$ and $M\in Z(\tau,S)$ such that
$h^i(\Ii _M(k-1))=0$, $i=0,1$. For any $p\in S$ we take a general arrow of $H$ (not of $\PP^n$ with $p$ as its reduction and
set $v_S:= \cup _{p\in S} v_p$. Set $W:= M\cup T\cup v_S$. By Remark \ref{ma6.0} we have $W\in Z'(\epsilon)$. Thus it is
sufficient to prove $h^1(\Ii _W(k))=0$. Since $\Res _H(W) =M$ and $h^1(\Ii _M(k-1))=0$ by the definition of $R(n,k-1)$, it is
sufficient to prove
$h^1(H,\Ii _{H\cap H,H}(k)) =0$. For any $p\in S$ the Zariski tangent space of $Y\cup v_S$ at $p$ is the tridimensional linear
space $M_p$ by the union of $v_p$ and the tangent lines at $p$ of the two irreducible components of $M$ containing $p$. For a
general $Y$ the linear space $E_p:= M_p\cap H$ has dimension $2$ and the degree $3$ zero-dimensional scheme $(2p,E_p)$ is the
connected component of $W\cap H$ with $p$ as its reduction (use that $Y\cap T=\emptyset$). Set $A:= \cup _{p\in S} (2p,E_p)$. 
Since $W\cap H = A\cup T \cup (Y\cap (H\setminus S\cup T))$, it is sufficient to prove $h^1(H,\Ii _{T\cup A\cup (Y\cap
(H\setminus S\cup T)),H}(k))=0$.

\quad {\bf Claim 1:} $h^1(H,\Ii _{T\cup A,H}(k))=0$.

\quad {\bf Proof of Claim 1:} We have $h^1(\Ii _{T,H}(k))=0$. To conclude the proof use the Proof of Claim 1 in the proof of
Lemma
\ref{rnn12}.

We order the points $p_1,\dots ,p_\beta$ of
$S$. For any
$b\in
\{0,\dots ,\beta\}$ set $A_b:= \cup _{i\le b} (2p_i,E_{p_i})$ with the convention $A_0=\emptyset$. Let $b$ the maximal $b\le
\beta$ such that $h^1(\Ii _{T,H}(k))=0$. Assume $b<\beta$. We work backwards. We use that $\beta \le 2k$ and hence the we only
need to prove that $h^1(H,\Ii _{T\cup Z}(k)) =0$, where $Z$ is a general union of $2k$ $2$-planes. Assume for the
moment $k>3$ and to have done the case $k=3$. Call
$T_3$ the curve obtained for $\Oo_H(3)$ and $W_3$ the $6$ general planar singularities.  Now assume $k\ge 4$. We dismantle $T$
in the following way. We add at each step
$t\Longrightarrow t+1$,
$t=3,\dots k-1$ a reducible curve $T_t$ and two $2$-planar points $W_t$ so that $T_3\cup \cdots \cup T_k$ is a specialization
of the family of curves with $\beta$ singularities of which $T$ is a general member. Set $T':= T_3\cup \cdots \cup T_k$
and $W':= W_3\cup \cdots W_k$. By semicontinuity it is sufficient to construct $T_4,W_4,\dots ,T_k,W_k$ such that $h^1(H,\Ii
_{T'\cup W',H}(k))=0$. Let $M \subset H$ be a hyperplane of $H$. We need  to apply $t-3$ residual exact sequences in which both terms on the
left and on the right have $h^1$ vanishing. To obtain this we need to allow to loose something at each step, but at most
$2t+2$ in the step $t\Longrightarrow t+1$. Thus we conclude by Lemma \ref{hor2}. To do the case $k=3$ we start with $K\in
Z(H;e;b_1,c_1;\dots ;b_e,c_e)$ with $h^1(H,\Ii _{K,H}(2)) =0$ and $h^0(H,\Ii_{K,H}(2))\le 1$ and add in a hyperplane $M$ of $H$
the union $F$
$6$ general planar $2$-points. We have $h^1(M,\Ii_{F,M}(3)) =0$. To have $h^1(M,\Ii _{(M\cap K)\cup F,M}(3)) =0$ and hence by
the residual exact sequence of $M$ in $H$ to get $h^1(H,\Ii _{M\cup F}(3))=0$ it would be sufficient to apply \cite[Lemma
2.1]{be7} if we have the necessary condition $\deg (K)\le h^0(M,\Ii_{F,M}(3))$, i.e. if $\deg (K) \le \binom{n+1}{3}-18$.
We have $\deg (K) =\lfloor \binom{n+1}{2}/3\rfloor$ if $c_i=0$ for all $i$.

By Claim 1 to get $h^1(H,\Ii _{W\cap H,H}(k)) =0$ it is sufficient to quote \cite{be6} as in step (a).

\subsection{$h^0(\Ii _{X''}(k-1))=0$} Now we prove that $h^0(\Ii _{X''}(k-1)) =0$ for a general $X''\in Z(\epsilon)$. Fix a
general
$S\subset H$ with $\#S=\beta$. By Lemma \ref{rma1} and
$R(n,k-1)$ there is an admissible generalized numerical set
$\tau = (n;s;b_1,e_1;\dots ;b_s,e_s)\prec \epsilon$ with $w_{k-1}(\tau) =\binom{n+k-1}{n}$ and $W\in Z(\tau,S)$ such that
$h^i(\Ii _W(k-1))=0$, $i=0,1$. For any $p\in S$ let $v_p$ be a general arrow of $\PP^n$. Set $W':= W\cup _{p\in S} v_p$.
By Remark \ref{ma6.0} $W'\in Z'(\eta)$. Since $h^0(\Ii _W(k-1)) =0$, we have $h^0(\Ii_{W'}(k-1))=0$. By the semicontinuity
theorem for cohomology we have $h^0(\Ii _U(k-1)) =0$ for a general $U\in Z(\eta)$. By our definition of $\prec$ there is
$X\supseteq U$ with $X\in Z'(\epsilon)$. Thus $h^0(\Ii _X(k-1)) \le h^0(\Ii _U(k-1))=0$. Then we continue as in the $h^1$-case.

\providecommand{\bysame}{\leavevmode\hbox to3em{\hrulefill}\thinspace}

\end{document}